\documentclass{amsart}
\usepackage{physics,comment}
\usepackage{amsmath}
\usepackage{amssymb, amsthm}
\usepackage{dsfont} %allows for \mathds{1}
\usepackage{enumitem} %easier changing of bullet points/enum lists, etc
\usepackage{graphicx} %allows for including graphics
\usepackage{hyperref}
\usepackage[nocompress,nospace,noadjust]{cite}
\usepackage{tikz}
\usepackage{pgfplots}
\pgfplotsset{compat=1.18}%to remove warning, good to fix the version so that spacing of labels stays the same
\usepackage{epsfig}
\usepackage{caption}
\usepackage[labelformat=brace]{subcaption}
\usepackage{environ}
\usepackage{ifthen}
\usepackage{verbatim}

\newcommand{\set}[1]{\left\{#1\right\}}
\newcommand{\PP}{\mathbb{P}}
\numberwithin{equation}{section}

\allowdisplaybreaks
\newtheorem{theorem}{Theorem}[section]
\newtheorem{lemma}[theorem]{Lemma}
\newtheorem{proposition}[theorem]{Proposition}
\newtheorem{corollary}[theorem]{Corollary}

\theoremstyle{definition}
\newtheorem{definition}[theorem]{Definition}
\newtheorem{assumption}[theorem]{Assumption}

\newtheorem{remark}[theorem]{Remark}

\newtheoremstyle{named}{}{}{\itshape}{}{\bfseries}{.}{.5em}{\thmnote{#3}}
\theoremstyle{named}

\newcommand{\R}{\ensuremath{\mathbb{R}}}

\newcommand{\RR}{\ensuremath{\mathbb{R}}}
\newcommand{\EE}{\ensuremath{\mathbb{E}}}
\newcommand{\Var}{\operatorname{Var}}

\newcommand{\comp}{\textup{Comp}}
\newcommand{\incomp}{\textup{Incomp}}
\newcommand{\LCD}{\operatorname{LCD}}

\newcommand{\ceil}[1]{\lceil #1 \rceil}

\newboolean{showDetails}

\setboolean{showDetails}{false} % toggle details here!

\NewEnviron{details}{%
    \ifthenelse{\boolean{showDetails}}{%
        \begingroup\color{blue}\textbf{Details:} \BODY \textbf{ End of details.}\endgroup %
    }{%
        \relax%
    }%
}%

\title[Gaps between Singular Values]{Gaps between Singular Values of Sample Covariance Matrices}
\author[N. J. Christoffersen]{Nicholas J. Christoffersen}
\address{Department of Mathematics\\ University of Colorado\\ Campus Box 395\\ Boulder, CO 80309-0395\\USA}
\email{nicholas.christoffersen@colorado.edu}

\author[K. Luh]{Kyle Luh}
\address{Department of Mathematics\\ University of Colorado\\ Campus Box 395\\ Boulder, CO 80309-0395\\USA}
\email{kyle.luh@colorado.edu}
\thanks{K. Luh has been partially supported by Simons Grant MP-TSM-00001988.}

\author[S. O'Rourke]{Sean O'Rourke}
\address{Department of Mathematics\\ University of Colorado\\ Campus Box 395\\ Boulder, CO 80309-0395\\USA}
\email{sean.d.orourke@colorado.edu}
\thanks{S. O'Rourke has been partially supported by NSF CAREER grant DMS-2143142.}

\author[C. Shearer]{Calum Shearer}
\address{Department of Mathematics\\ University of Colorado\\ Campus Box 395\\ Boulder, CO 80309-0395\\USA}
\email{calum.shearer@colorado.edu}

\date{\today}

\begin{document}
\begin{abstract}
    We study the gaps between consecutive singular values of random rectangular matrices.  Specifically, if $M$ is an $n \times p$ random matrix with independent and identically distributed entries and $\Sigma$ is an $n \times n$ deterministic positive definite matrix, then under some technical assumptions we give lower bounds for the gaps between consecutive singular values of $\Sigma^{1/2} M$.  As a consequence, we show that sample covariance matrices have simple spectrum with high probability.  Our results resolve a conjecture of Vu [{\em Probab. Surv.}, 18:179--200, 2021].  We also discuss some applications, including a bound on the spacings of eigenvalues of the adjacency matrix of random bipartite graphs.  
\end{abstract}

\maketitle

\section{Introduction}

The gaps between eigenvalues have been objects of significant study in random matrix theory.  A famous example is the Wigner surmise for the density of level spacings of normalized eigenvalues \cite{MehtaRMT,wignersurmise}.  Gaps between eigenvalues have been studied for a variety of matrix models in the context of gap distributions, extreme gap statistics, and overcrowding (Wegner) estimates; we refer the reader to \cite{MehtaRMT,tao2010randommatricesuniversalitylocal,erdos2014gapuniversalitygeneralizedwigner,tao2012asymptoticdistributionsingleeigenvalue,MR2917064,MR4416591,MR3112927,MR3231045,MR3741950,MR2868046,MR2587574,MR2851058,MR2536111,MR4009717,MR3699468,MR3695802} and references therein. 

A square matrix is said to have \emph{simple spectrum} if its eigenvalues all have (algebraic) multiplicity one. This was established for a class of Wigner matrices by Tao and Vu in \cite{tao2014randommatricessimplespectrum}.  In particular, the results hold for the adjacency matrix of Erd\H{o}s--R\'{e}nyi random graphs, which addressed a long-standing conjecture of Babai and is related to the graph isomorphism problem \cite{10.1145/800070.802206}.  
The gaps between complex eigenvalues of non-Hermitian random matrices with independent entries were studied in \cite{ge2017eigenvaluespacingiid}; these results were improved upon in \cite{luh2020eigenvectorscontrollabilitynonhermitianrandom}.  
More recently, it has been conjectured by Vu in \cite{vu2020recentprogresscombinatorialrandom} that the singular value spectrum for various random matrix models are simple with high probability.  

Related to this is the study of the {minimum} gap between any two eigenvalues. Ben Arous and Bourgade found the limiting distribution of the minimum gap for the Gaussian Unitary Ensemble \cite{MR3112927}.  In \cite{nguyentaovugaps}, Nguyen, Tao and Vu showed that the minimum gap between eigenvalues of a random \(n \times n\) Wigner matrix with subgaussian entries is greater than \(n^{-3/2-o(1)}\) with high probability, where \(o(1)\) denotes a quantity tending to zero as \(n\) tends to infinity.  
These results have also been extended to a class of sparse symmetric matrices by Lopatto, the second author and Vu in \cite{luhvusparse, luh_lopatto_sparse_gaps}.  Recent progress has also been made by Han in \cite{han2024smallballprobabilitymultiple} on bounding  the probability that the eigenvalues of Wigner matrices satisfy a fixed linear equation.

In this paper, we study the spacings of singular values of a random matrix $\Sigma^{1/2}M$, where $M$ is an $n \times p$ random rectangular matrix with independent entries and $\Sigma$ is an $n \times n$ deterministic positive definite matrix.   The squares of the singular values are equal to the eigenvalues of the sample covariance matrix \(M^\mathrm{T} \Sigma M\), and the block matrix \(\begin{pmatrix} 0 & \Sigma^{1/2} M \\ M^\mathrm{T} \Sigma^{1/2} & 0 \end{pmatrix}\) has eigenvalues equal to the singular values of \(M\) and their negatives. We demonstrate analogous results to those of \cite{nguyentaovugaps} for singular values of $\Sigma^{1/2} M$, under various assumptions on the matrix $\Sigma$ and the entries of $M$.  In particular, our results resolve a conjecture of Vu from \cite{vu2020recentprogresscombinatorialrandom} that the singular value spectrum of a square random matrix with independent, identically-distributed entries is simple with high probability. 
Our proof adapts the general strategy from \cite{nguyentaovugaps}.  However, the singular value equations introduce several significant conceptual difficulties along with numerous  technical obstacles that do not arise in \cite{nguyentaovugaps}.  

In an independent and concurrent note \cite{han2025singularvalues}, Han addresses the same conjecture and provides a stronger probability bound.  However, his results require a subgaussian assumption on the random variables and do not provide quantitative estimates of the gap sizes.  

\subsection{Notation}

All constants that appear in our results are independent of the dimensions \(n\) and \(p\). All constants, often denoted as $C$ and $c$, depend on the constants \(c_{\ref{matrix_model}}, K,L\) and \(m_4\) (introduced below) unless otherwise specified.  

For an $n \times p$ matrix $A$ with $p \leq n$, $\sigma_1(A) \geq \cdots \geq \sigma_p(A) \geq 0$ denote the ordered singular values of $A$, i.e., the eigenvalues of $\sqrt{A^\mathrm{T} A}$, where $A^{\mathrm{T}}$ is the transpose of $A$. A right singular vector of \(A\) is a (unit) eigenvector of \(A^\mathrm{T}A\), and a left singular vector is a (unit) eigenvector of \(A A^\mathrm{T}\), with eigenvalue equal to one of \(\sigma^2_1(A) \geq \dots \geq \sigma^2_p(A)\).  Note that $A A^{\mathrm{T}}$ may have more eigenvalues than $A^{\mathrm{T}}A$, but we ignore those trivial eigenvalues that are deterministically zero.  We often use $I$ to denote the identity matrix.

Asymptotic notation (e.g., \(o\), \(O\), \( \Omega \)) is used as \(n\) tends to infinity, so we say that \(f(n)=O(g(n))\) if \(|f(n)| \leq C g(n)\) for all \(n\) sufficiently large and for some constant \(C\) sufficiently large, independent of \(n\). We say that \(f(n) = o(g(n))\) if \(\lim_{n \to \infty} f(n)/g(n) = 0\). We say that \( f(n) = \Omega(g(n)) \) if \( f(n) \geq C g(n) \) for all \(n\) sufficiently large and for some constant \(C\) sufficiently large, independent of \(n\).

For a real number \(a\), \(\lceil a \rceil\) denotes the least integer greater than or equal to \(a\) and \(\lfloor a \rfloor\) denote the greatest integer less than or equal to \(a\). Given a finite set \(A\) we let \(\abs{A}\) denote the cardinality of \(A\). If \(\mathcal{A}\) and \(\mathcal{B}\) are events, \(\mathcal{A} \vee \mathcal{B}\) denotes their union and \(\mathcal{A} \wedge \mathcal{B}\) denotes their intersection. For a vector \(x = (x_1,x_2,\dots,x_m) \in \mathbb{R}^m\), its norm, \(\norm{x}\), is the Euclidean \(2\)-norm: \(\norm{x} = \left(\sum_{i=1}^m x_i^2\right)^{1/2}\) and \(S^{m-1}\) is the Euclidean sphere \(\left\{ x \in \mathbb{R}^m \colon \norm{x} = 1 \right\}\). If \(J = \{i_1,\dots, i_k\} \subset \set{1,2,\dots, m}\), then we define the vector \(x_{J} = (x_j)_{j \in J}\in \R^{J}\). For a positive integer \( m \), we define \( [m] \) to be the discrete interval \( \left\{1, 2, \ldots, m\right\} \).

Finally, for a matrix \(A\), we define its spectral (operator) norm
\[
    \norm{A} = \sup_{\norm{x}=1} \norm{Ax},
\]
where \(\norm{Ax}\) is the Euclidean \(2\)-norm. 
\subsection{The model and assumptions}
We will consider the random matrix $\Sigma^{1/2}M$, where $M$ is an $n \times p$ random rectangular matrix with independent entries and $\Sigma$ is an $n \times n$ deterministic positive definite matrix.   We make the following assumptions concerning $M$ and $\Sigma$.

\begin{assumption}\label{matrix_model}
    \(M\) is an \(n \times p\) random matrix, where \(p = \lambda n\) and the aspect ratio \(\lambda\) satisfies \begin{equation}
        c_{\ref{matrix_model}} \leq \lambda \leq 1
    \end{equation} for some constant \(c_{\ref{matrix_model}} > 0\).  In particular, this implies that $p \leq n$.  We also make the assumption that the entries of \(M\) are independent and identically distributed copies of the real-valued random variable \(\xi\), where \(\xi\) has mean zero, unit variance, and satisfies
    \begin{equation}
        \mathbb{E}[\xi^4] \leq m_4
    \end{equation}
    for some absolute constant \(m_4 < \infty\) (note that H\"older's inequality implies that \(\EE{\abs{\xi}^3} \leq m_4^{4/3}\)). The random variable $\xi$ is called the \emph{atom variable} of $M$.
    Furthermore, \(\xi_{ij}\) will be used to denote the \((i,j)\)-th entry of \(M\).
\end{assumption}

\begin{assumption}\label{opnorm-assumptions}
\(\Sigma\) is an \(n \times n\) deterministic positive definite matrix which satisfies
\begin{equation}\label{eq:M-sigma-assumption}
    \norm{\Sigma}, \norm{\Sigma^{-1}} \leq L^2
\end{equation}
for some absolute constant \(L \geq 1\). It follows that \(\Sigma^{1/2}, \Sigma^{-1/2}\) are well-defined and have norms bounded above by \(L\).  
\end{assumption}
Regarding the bounds on \(\lambda\) in the above assumptions, we note that the non-trivial singular values of $M$ and $M^{\mathrm{T}}$ are the same, so the assumption that $c_{\ref{matrix_model}} \leq \lambda \leq 1$ can be replaced with the bound $c_{\ref{matrix_model}} \leq \lambda \leq c_{\ref{matrix_model}}^{-1}$ without further explanation in the case where \( \Sigma = I \), where $I$ denotes the identity matrix. For simplicity, we restrict to the case when $\lambda$ is bounded above by one. 

In the case where \( \Sigma \neq I \), again, our results hold for $c_{\ref{matrix_model}} \leq \lambda \leq c_{\ref{matrix_model}}^{-1}$ but the proof for $\lambda \geq 1$ requires some superficial and routine changes to the proof.  See Remark \ref{rem:extension}.

    We define the event \(\mathcal{E}_K\) to be
    \begin{equation}
        \label{def:eK}
   \mathcal{E}_K := \set{\norm{M} \leq K \sqrt{n}} \ \wedge \ \set{\sigma_p(M) \neq 0},
   \end{equation} for \(K\geq1\) a constant. 
    
We will state some results in the case when the atom variable $\xi$ is subgaussian.  Recall that a random variable \(\xi\) is called \textit{subgaussian} if there exists a constant \(q > 0\) (called the \emph{subgaussian moment} of \(\xi\)) such that 
\[
\mathbb{P}\set{\abs{\xi}>t} \leq \frac{1}{Q}e^{-Q t^2}
\]
for all $t > 0$.  

\subsection{Main results}

We now state our main result.  

\begin{theorem}[Main result] \label{thm:main}
    Let \(M\) and \(\Sigma\) be as in Assumptions \ref{matrix_model} and \ref{opnorm-assumptions}, and assume $K \geq 1$. Then there exist constants \(c, C_{\ref{main1}} > 0\) (depending on \(c_{\ref{matrix_model}}, K,L\) and \(m_4\)) such that, for any \(n^{-c} \leq \alpha \leq c\) and \(\delta \geq n^{-c/\alpha}\), 
\begin{equation} \label{eq:main_bound}
\sup_{1 \leq i \leq p-1} \PP \left( \sigma_{i}^2(\Sigma^{1 / 2}M) - \sigma_{i+1}^2(\Sigma^{1 / 2}M) \leq  \sigma_{i+1}(\Sigma^{1 / 2}M) \delta n^{-1/2}  \wedge \mathcal{E}_{K} \right) \leq C_{\ref{main1}} \frac{\delta}{\sqrt{\alpha}} 
\end{equation}
where $\mathcal{E}_K$ is defined in \eqref{def:eK}.
\end{theorem}

A few comments concerning Theorem \ref{thm:main} are in order.  First, Theorem \ref{thm:main} is the analogue of Theorem 2.1 from \cite{nguyentaovugaps}.  Unlike the results from \cite{nguyentaovugaps}, we see that the $i$-th gap between the squared singular values depends on the location of the $(i+1)$-st singular value. Figures \ref{figure:whole spectrum} and \ref{figure: zoom ins}, corresponding to our model with \(\Sigma = I, n=2500\), and Rademacher atom variables, demonstrate this dependence on the location of the $(i+1)$-st singular value, showing that the squared singular values are much more crowded near the ``hard edge'' of the spectrum at \(0\) than in the bulk or near the soft edge at \(4 n\). Heuristically, this can be explained in the following way.  After appropriately scaling, the empirical measure constructed from the squared singular values is known to converge to the Marchenko--Pastur density as $p, n \to \infty$ such that the ratio $p/n \to \eta \in (0, 1]$.  When $\eta = 1$, the Marchenko--Pastur distribution is unbounded at the hard edge, and as such one expects the spacing of singular values at the edge to be different than the spacing in the bulk, where the density is bounded. The appearance of $\sigma_{i+1}(\Sigma^{1 / 2}M)$ on the right-hand side of \eqref{eq:main_bound} quantitatively captures this behavior. 

\begin{center}
    \begin{figure}[htbp]
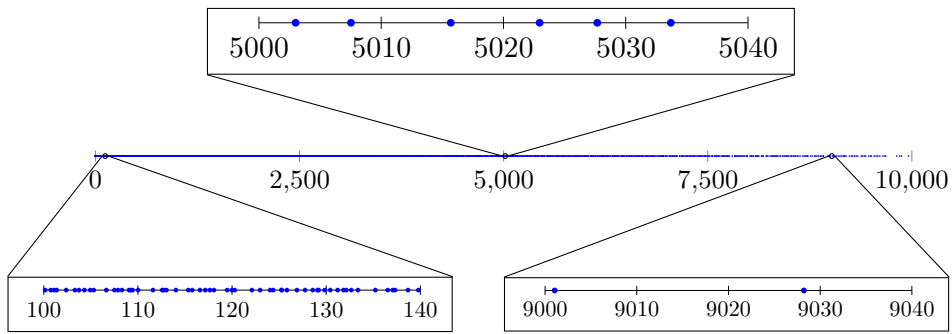
 
	\centering
	\begin{subfigure}{\linewidth}
 
	\end{subfigure}
	\caption{A numerical simulation of the spacings of squared singular values of a \(2500 \times 2500\) random matrix with Rademacher atom variable and \(\Sigma=I\).}
    \label{figure:whole spectrum}
\end{figure} % spectrum with zoom

\begin{figure}[htbp]
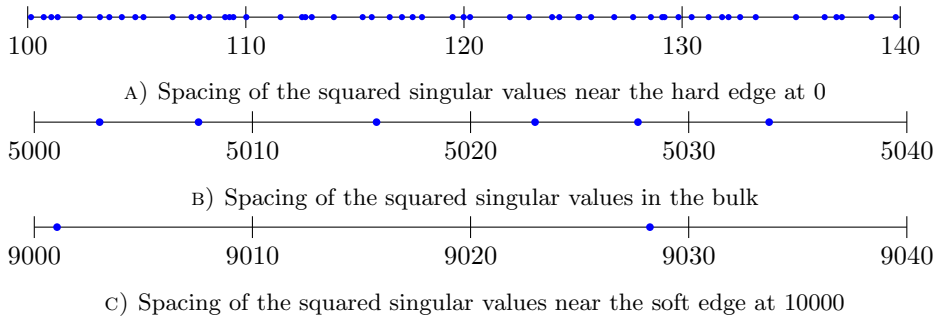
 
\begin{subfigure}{\linewidth}
\noindent%
  \caption{Spacing of the squared singular values near the hard edge at \(0\)}
 % hard edge zoom
\end{subfigure}

  \begin{subfigure}{\linewidth}
\noindent %
  \caption{Spacing of the squared singular values in the bulk}
 % bulk zoom
\end{subfigure}

\begin{subfigure}{\linewidth}
\noindent %
\caption{Spacing of the squared singular values near the soft edge at \(10000\)}
\end{subfigure}
  \caption{Enlarged figures showing a numerical simulation of the spacings of the squared singular values of a \(2500 \times 2500\) random matrix with Rademacher atom variable and \(\Sigma=I\) at different parts of the spectrum.}
  \label{figure: zoom ins}
\end{figure} % soft edge zoom
\end{center}

Second, our proof methods extend to random matrix models \(M\) whose entries have a non-zero mean, provided that either the matrix of means \( \mathbb{E} M \) is rank one or is of small norm. For details, see Remark \ref{rem:mean non zero} and Remark \ref{rem:gen_main1}.

Third, although our results may be true without any higher moment assumptions, our proof method needs the fourth moment assumption to guarantee that \( \mathbb{P}(\mathcal{E}_K^{c}) \) is small (see, for instance, \cite[Theorem 5.8]{Bai2010}), which is an unavoidable ingredient in proofs that involve $\varepsilon$-nets.
With the bounded fourth moment assumption, we may remove $\mathcal{E}_K$ from the event specified in \eqref{eq:main_bound}. A theorem of Bai--Yin  \cite[Theorem 1]{BaiYin} implies that under Assumption \ref{matrix_model} there exists a constant \(K \geq 1\) such that
\begin{align*}
    \mathbb{P}\set{\norm{M} > K \sqrt{n}} = o(1), 
\end{align*}
where \(K\) depends on \(c_{\ref{matrix_model}}\). Moreover \cite[Theorem 1.1]{MR4255145} states that there exists a small constant \( c > 0 \) such that
\[
    \mathbb{P}\set{\sigma_p(M) = 0} \leq 2 e^{-cn}.
\]
This gives that \(\mathbb{P}(\mathcal{E}_K^c) = o(1)\) and allows us to derive the following corollary.

\begin{corollary}
    \label{general simple}
    Let \(M\) and \(\Sigma\) satisfy Assumptions \ref{matrix_model} and \ref{opnorm-assumptions}. Then with probability \(1-o(1)\), all singular values of \(\Sigma^{1/2}M\) are distinct. In particular, the sample covariance matrix $M^{\mathrm{T}} \Sigma M$ has simple spectrum with probability $1 - o(1)$.  
\end{corollary}

When the atom variable $\xi$ is subgaussian, we can give a more quantitative bound, using stronger results on the probability of \(\mathcal{E}_K^c\) that are standard in the literature. For example, it follows from \cite[Theorem 4.4.5]{vershynin-hdp} and \cite[Theorem 1.1]{rudelson2009smallestsingularvaluerandom} (or \cite[Theorem 1.1]{MR4255145}) that in the case of \(\xi\) being subgaussian we have that
\[
\mathbb{P}(\mathcal{E}_K^c) \leq 2 \exp(-cn)
\]
for some constant \(c\) depending on the subgaussian moment of \( \xi \). This yields the following corollary to our main result (further details are given in Section \ref{mainproofs}).

\begin{corollary}\label{subgauss simple}
    Let \(M\) and \(\Sigma\) satisfy Assumptions \ref{matrix_model} and \ref{opnorm-assumptions} with subgaussian atom variable $\xi$. Then there exists a positive constant \(c\) depending on \(c_{\ref{matrix_model}}, L, K\) and the subgaussian moment \(q\) such that with probability at least \(1 - 2 \exp(-n^{c})\) all singular values of \(\Sigma^{1/2} M\) are distinct. 
\end{corollary}

\begin{remark}
    In \cite{han2025singularvalues}, under the subgaussian assumption on the random variables and setting $\Sigma = I$, Han shows the stronger result that the probability that the singular value spectrum is simple occurs with probability at least $1- \exp(-cn)$ for some constant $c > 0$.   
\end{remark}

The next corollary follows by choosing \(\alpha\) to be a small constant in Theorem \ref{thm:main}.  
\begin{corollary}
    Let \(M\) and \(\Sigma\) satisfy Assumptions \ref{matrix_model} and \ref{opnorm-assumptions}. If we define
    \[\delta_{\min} = \min_{1 \leq i \leq p-1} \left\{\sigma_{i}(\Sigma^{1/2}M) - \sigma_{i+1}(\Sigma^{1/2}M) \right\},\] 
    then 
    \[
    \delta_{\min} \geq n^{-3/2-o(1)}.
    \]
    with probability \(1 - o(1)\).
\end{corollary}
The justification of the above corollary can be found within the proof of Corollary \ref{cor:sing_value_diff}.

\subsection{Proof Strategy}\label{proof strategy} 
In order to effectively demonstrate the main ideas of our proof, we describe the proof strategy assuming \(\Sigma\) is the identity matrix.  We present the proof of the general case in the bulk of the paper. 

Our overall strategy is motivated by that of \cite{nguyentaovugaps}. We begin by decomposing the matrix $M$ as 
\[
M = \begin{pmatrix} M' \mid  X \end{pmatrix},
\]
where \(X\) is the final column of \(M\). Note that by our assumptions, the entries of \(X\) are independent from \(M'\). We now manipulate the eigenvalue-eigenvector equation \(M^\mathrm{T} M v = \sigma_i^2(M) v\) in order to relate the gap size \(\sigma_i^2(M) - \sigma_{i+1}^2(M)\) to the inner product of singular vectors of the minor \(M'\) with the random vector \(X\). Indeed, if one decomposes \(v\) as 
\[
v = \begin{pmatrix}
    v' \\
    b
\end{pmatrix}
\]
where \(v' \in \mathbb{R}^{p-1}\) and \(b \in \mathbb{R}\)
then the eigenvalue-eigenvector equation becomes 
\[
M^\mathrm{T} M v = \begin{pmatrix}
    {M'}^\mathrm{T} M' & {M'}^\mathrm{T} X \\
    X^\mathrm{T} M' & X^\mathrm{T} X 
\end{pmatrix}
\begin{pmatrix}
    v' \\
    b
\end{pmatrix}
= \sigma_i^2(M) \begin{pmatrix}
    v' \\
    b
\end{pmatrix},
\]
the first \((n-1)\) rows of which read as
\[
M'^{\mathrm{T}} M' v' + M'^{\mathrm{T}} X b = \sigma_i^2(M) v'.
\]
By multiplying this equation on the left by \( u^{\mathrm{T}} \), where \( u \) is a left singular vector of \( M' \) corresponding to the \( i \)-th largest singular value, and then rearranging terms, we find that if \( u \) is a left singular vector of \( M' \), the following holds:
\[
\abs{u^{\mathrm{T}} X} \leq \frac{\sigma_i^2(M) - \sigma_i^2(M')}{\sigma_i(M') |b|} \leq  \frac{\sigma_i^2(M) - \sigma_{i+1}^2(M)}{\sigma_i(M')|b|} ,
\]
where the final inequality follows from Cauchy's interlacing law. Therefore, the problem of bounding the gaps between squared singular values from below can be reduced to bounding the probability that $\abs{u^{\mathrm{T}} X}$ is small (along with showing that, on average, a coordinate of a singular vector \(v\)\ is sufficiently large). This probability is known as a \textit{small ball probability} and related to the phenomena of \textit{anti-concentration}. Importantly, our argument ensures that the random vector \(X\) is independent of the entries of the singular vector \(u\). The proof reduces to controlling the small ball probability of a typical singular vector.

The main technical hurdle arises in the process of excluding potential singular vectors. One could attempt to show that the probability \(\mathbb{P} \left\{\norm{M^{\mathrm{T}}Mv - a} \leq \varepsilon \sqrt{n} \right\}\) is small for any \(a \in \mathbb{R}^p\) in order to demonstrate that for fixed \(v\) and \(\sigma^2\), it is unlikely that \(M^{\mathrm{T}} M v= \sigma^2 v\). However, the sample covariance matrix \(M^{\mathrm{T}} M\) does not have independent entries or even independent rows or columns, which is a serious complication. We deal with this issue by controlling the event that \(\|Mx - a\| \leq \varepsilon \sqrt{n}\) and \(\|M^{\mathrm{T}}y -b\| \leq \varepsilon' \sqrt{n}\) \textit{simultaneously}. This can be thought of as a discretization of the singular value equations
\[
Mv = \sigma u \text{ and } M^\mathrm{T} u = \sigma v.
\]
Our main technical innovations occur in the handling of these simultaneous events. 

The inclusion of the covariance matrix \(\Sigma\) does not change any of the main mathematical ideas of our proofs, but leads to slightly more involved arguments. Our covariance matrix \(\Sigma\) is assumed to have bounded operator norm, independent of \(n\) and \(p\) in order to successfully use our \(\varepsilon\)-net arguments.  

\subsection{Applications}\label{applications}
The gap between eigenvalues of random matrices is important in applications, particularly to numerical linear algebra (e.g., \cite{nguyentaovugaps, luh_lopatto_sparse_gaps}). When working with rectangular and non-symmetric matrices, one is concerned with \emph{singular value gaps}, rather than eigenvalue gaps \cite{KUHL2024109022}. Even in the case of symmetric matrices, singular value gaps can play an important role. For instance, the power method (also known as power iteration) is an algorithm to find the largest eigenvalue (by absolute value) and its corresponding eigenvector. The convergence rate of the error given by the power method is exponential, with base given by the ratio \(\sigma_2 / \sigma_1 = 1 - \frac{\sigma_1 - \sigma_2}{\sigma_1}\). Power iteration and its related ideas (e.g., Krylov subspaces, Lanczos Algorithm) are still an area of active research \cite{Yuan_RB_Lanczos}. In particular, variants of the Lanczos Algorithm (whose convergence also depends on gaps of subsequent singular values \(\sigma_{k+1} - \sigma_k\)) are found in standard numerical linear algebra programs in MATLAB, Mathematica, Scipy (python), and others through their dependency on Fortran's ARPACK library \cite{lehoucq1998arpack}. In particular, these methods and other variants are found in various implementations of principal component analysis (PCA) and singular value decomposition (SVD) algorithms \cite{Musco_Rand_Blk_Krylov_SVD_Power_Method, Zhang2020_EIF_Analysis_SVD}, as well as neural network architectures and computations \cite{pmlr-v139-li21u, Peng2022_Graph_Convol_Paradigm}. Power iteration also appears in recommendation engines widely used across the internet \cite{twitter_power_method, google_page_rank}. Even modern, machine learning-based recommendation engines use SVD algorithms that assume simplicity of singular values \cite{Peng2022_Graph_Convol_Paradigm}.

As another example, we examine the graph isomorphism problem (GI), which asks if two given graphs \(G, G'\) are isomorphic. Famously, it is not known if GI is solvable in polynomial time, nor if it is NP-complete \cite{MR3966534}. However, there are several polynomial-time algorithms to determine if \(G\) and \(G'\) are isomorphic assuming that the eigenvalues of the adjacency matrices are simple or of bounded multiplicity \cite{leighton1979certificates, 10.1145/800070.802206, Klus_spectral_assignment_graph_iso}. Let $A$ and $B$ be the adjacency matrix of \(G\) and \(G'\) respectively. To convey the intuition for the algorithm in \cite{leighton1979certificates}, suppose that the entries of the eigenvectors of $A$ and $B$ corresponding to $\lambda_1$ have distinct magnitudes; then this specifies the only possible mapping between the nodes of $G$ and $G'$. The actual algorithm needs to proceed more carefully, accounting for entries with identical magnitudes in the eigenvectors and iteratively partitioning the nodes into possible isomorphism classes.

The tractability of GI for graphs with simple spectrum suggests a natural question, posed by Babai \cite{10.1145/800070.802206}, of whether most $n$ vertex graphs fall into this class. This was answered in the affirmative in \cite{tao2014randommatricessimplespectrum,nguyentaovugaps}. Quantitatively, in \cite{nguyentaovugaps}, the authors prove that of the $2^{\binom{n}{2}}$ simple graphs on $n$ vertices, all but a $\delta_n$ proportion of the graphs have simple spectrum, where $\delta_n = O(n^{-C})$ for any constant $C$. However, it is often of interest to restrict GI to a subset of graphs with certain combinatorial properties.   
For example, GI is known to be solvable in polynomial time in the case of planar graphs \cite{Hopcroft1972}, graphs with bounded degree \cite{Luks1982} and others \cite{MR3966534}. GI, and more generally the subgraph isomorphism problem, restricted to bipartite graphs is of recent interest \cite{UEHARA2005479, 8806332}. Since the number of bipartite graphs that span two sets of size $n$ is exactly $2^{n^2}$, which is an exponentially small fraction of all possible graphs on these vertices, the result in \cite{nguyentaovugaps} does not provide any information on the proportion of bipartite graphs with simple spectrum.

Note that the adjacency matrix of a bipartite graph is of the form \(\begin{pmatrix} 0 & M \\ M^{\mathrm{T}} & 0 \end{pmatrix}\), so that the non-zero eigenvalues of this matrix are precisely \(\pm \sigma_i(M)\) (the non-zero singular values of \(M\) and their negations). Hence, if one is able to conclude that the singular values of \(M\) has no repeated singular values, then GI is solvable for the associated pair of graphs in polynomial time. 
\begin{theorem} \label{thm:bipartite}
   Of the \( 2^{n^2} \) bipartite graphs on \( n + n \) vertices, the proportion of graphs whose adjacency matrices have non-simple spectrum is at most \(2 \exp(-n^{c})  \) for some constant $c > 0$. 
\end{theorem} 

This theorem is not an immediate application of our main result, as the adjacency matrix of such a graph is of the form \( \begin{pmatrix} 0 & M \\ M^{\mathrm{T}} & 0  \end{pmatrix} \), where \( M \) has non-zero mean. In particular, the spectral norm of \( M \) is of order \( n \), rather than order \(\sqrt{n}\), as discussed in Remark \ref{rem:mean non zero}. However, we generalize our main result to include this class of random matrices in Remark \ref{rem:gen_main1}.

As a consequence, this shows that GI is solvable in polynomial time on all but a \( e^{-cn^{c}} \) proportion of bipartite graphs on \( n + n \) vertices.
\begin{corollary}
	There exists a constant $c > 0$ and a polynomial-time algorithm for bipartite graphs on $n+n$ vertices after excluding at most 
    \[  2\exp(-n^{c} ) 2^{n^2}\] graphs.  Phrased probabilistically, there exists a polynomial-time algorithm that resolves GI on two uniformly sampled bipartite graphs on $n+n$ vertices that succeeds with probability at least $1 - 2 \exp(-n^{c} )$.
\end{corollary}
\begin{proof}
	Theorem \ref{thm:bipartite} allows us to conclude that non-zero eigenvalues of the adjacency matrix are simple with the requisite probability.  The well-established results on the singularity of square Bernoulli matrices assure us that the adjacency matrix is non-singular (see, for example, \cite{tikhomirov2020singularity}).  Thus, any of the algorithms from \cite{leighton1979certificates, 10.1145/800070.802206, Klus_spectral_assignment_graph_iso} apply.  
\end{proof}
We remark that it is possible to use our methods to prove that the non-zero eigenvalues of the adjacency matrix for \emph{unbalanced} bipartite graphs are simple as well.  However, in this case, the null space of the adjacency matrix is deterministically large.  The algorithm in \cite{leighton1979certificates} can be modified by restricting attention to the eigenvectors corresponding to non-zero eigenvalues, but we do not pursue this algorithmic direction here.

\subsection{Outline}
The rest of the paper is devoted to the proofs of our main results.  We present some anti-concentration bounds in Section \ref{sec:anti-concentration}.  Technical bounds concerning compressible and incompressible singular vectors are presented in Sections \ref{compressible section} and \ref{incompressible section}, respectively.  Finally, these tools are combined in Section \ref{mainproofs} to complete the proofs. 

\section{Anti-concentration tools} \label{sec:anti-concentration}
We begin by stating some lemmas regarding the anti-concentration of random vectors, which will be used during our proofs. 

The L\'{e}vy concentration function is defined as
\begin{align*}
    \mathcal{L}(X, \varepsilon) := \sup_{a \in \RR^m} \PP \set{ \norm{X-a} \leq \varepsilon},
\end{align*}
where \(X\) is a random vector in \(\mathbb{R}^m\) and \(\varepsilon >0\). $\mathcal{L}(X, \varepsilon)$ bounds the \emph{small-ball probabilities} of $X$, i.e., the probabilities that $X$ falls in a ball of radius $\varepsilon$.  

We begin by showing that our assumptions on the boundedness of the \(4\)th moment of \(\xi\) implies that \(\xi\) obeys the following weak anti-concentration bound. This type of result is well-known, but we include a complete proof as we could not locate a statement with our exact assumptions in the literature.
\begin{lemma}\label{PZ}
Let \(\xi\) be a random variable with mean \(0\) and variance \(1\) satisfying \( \mathbb{E}[\xi^4] \leq m_4\)
for some constant \(m_4 < \infty\). Then for all \(\varepsilon \in (0,1)\), there exists \(p=p(\varepsilon,m_4)<1\) depending only on \(\varepsilon\) and \(m_4\) such that
    \begin{equation}
    \mathcal{L}(\xi,\varepsilon) \leq p    .
    \end{equation}
\end{lemma}
    \begin{proof}
    We seek to bound
    \(
        \sup_{a\in \R} \mathbb{P} \set{|\xi - a| \leq \varepsilon}.
    \)
    We do this by splitting into the case \(\abs{a} > 3\) and the case \(\abs{a} \leq 3 \), providing an upper bound uniform in \(a\) in each case. 
    For the range \(\abs{a} > 3\), observe that if \(\varepsilon \in (0,1)\) then \(\abs{\xi-a} < \varepsilon\) implies that \(\abs{\xi} \geq 2\). It follows from Markov's inequality and \(\EE{\abs{\xi}^2} =1\) that
    \begin{align*}
        \sup_{a\in \R\setminus [-3,3]} \mathbb{P}\set{|\xi - a| \leq \varepsilon}
        \leq \mathbb{P}\set{|\xi| \geq 2} \leq \frac{1}{4}.
    \end{align*}
    On the other hand, for \(\abs{a} \leq 3\), we apply the Paley-Zygmund inequality to the random variable \(\abs{\xi-a}^2\), noting that
    \[
    \mathbb{E}\abs{\xi-a}^2 = \mathbb{E}\abs{\xi^2} - 2a \mathbb{E}[\xi] + a^2 = 1+a^2,
    \]
    so that for any \(\varepsilon \in (0,1) \), 
    \begin{align}
       \nonumber \mathbb{P}\set{|\xi - a| \leq \varepsilon} &= 1 - \mathbb{P}\set{|\xi - a|^2 > \varepsilon^2} \\
       \nonumber &= 1 - \mathbb{P}\set{|\xi - a|^2 > \frac{\varepsilon^2}{1+a^2} \EE{ \abs{\xi-a}^2 }} \\
       \nonumber &\leq 1 - \left( 1 - \frac{\varepsilon^2}{1+a^2} \right)^2 \frac{ \big( \mathbb{E}\ \abs{\xi - a}^2 \big) ^2}{\mathbb{E}\abs{\xi - a}^{4}} \\
       \nonumber &\leq 1 - \left( 1 - \varepsilon^2 \right)^2 \frac{(1+a^2)^2}{\EE \abs{\xi-a}^4} \\
        &\leq 1 - (1-\varepsilon^2)^2 \frac{1}{\EE \abs{\xi-a}^4}. \label{eq:PZ}
    \end{align}
    Now we observe that for \(\abs{a} \leq 3\) we have 
    \begin{align*}
        \EE{\abs{\xi-a}^4} &= \EE{\abs{\xi}^4} - 4 \EE{\abs{\xi}^3}a + 6 \EE{\abs{\xi}^2}a^2 - 4 \EE{\abs{\xi}}a^3 + a^4 \\
        &\leq m_4 + 4 m_4^{3/4} \abs{a} + 6a^2 + a^4 \\
        &\leq m_4 + 12m_4^{3/4} + 54 + 81,
    \end{align*}
    an expression depending only on \(m_4\) which we can use to upper bound the denominator in \eqref{eq:PZ}, giving us an upper bound for \(\mathbb{P}\set{|\xi - a| \leq \varepsilon}\) depending only on \(\varepsilon\) and \(m_4\) when \(\abs{a} \leq 3\). Taking the maximum of this upper bound and \(1/4\) gives the required constant \(p = p(\varepsilon,m_4)\). 
    \end{proof}
Now, let us state the Tensorization Lemma (\cite[Lemma 3.4]{VershyninInvert}), which we will use throughout. This lemma allows one to use bounds on the L\'{e}vy concentration function of each individual coordinate of a random vector in order to bound the L\'{e}vy concentration function of the random vector. 

\begin{lemma}[Tensorization Lemma] \label{tensorization} 
    Let \(X = (X_1,\dots, X_m)\) be a random vector in \(\mathbb{R}^m\) with independent coordinates \(X_i\). Suppose there exist \(\varepsilon>0\) and  \(0 < p < 1\) independent of \(i\) such that
    \[
    \mathcal{L}(X_i,\varepsilon) \leq p
    \]
    for all \(i\). Then there exist \(\varepsilon_1>0\) and \(0 < p_1 < 1\) depending only on \(\varepsilon\) and \(p\) such that
    \[
    \mathcal{L}(X, \varepsilon_1 \sqrt{m}) \leq p_1^m.
    \]
\end{lemma}
We shall typically use the above lemma in the situation that \(X\) is a product of the form \(Mv\), where \(M\) is a random matrix, and \(v\) is a fixed unit vector. Also note that we will often write \(e^{-cm}\) for some \(c >0\) in place of \(p_1^m\) when applying this lemma. 

Note that \((Mv)_i = \sum_{j=1}^m \xi_{ij} v_j\). Therefore, we also often use the following lemma \cite[Lemma 3.3]{VershyninInvert} on the anti-concentration of sums.

\begin{lemma} \label{sum anticonc}
    Let \(\xi_1,\dots \xi_n\) be independent random variables with unit variances and centered \(4\)th moments uniformly bounded by \(m_4 < \infty\). Then for every \(0 < \varepsilon < 1\), there exists \(0 < p < 1\) depending only on \(\varepsilon\) and \(m_4\) such that for every \(x = (x_1, \dots, x_m) \in S^{m-1}\), the sum \(S = \sum_{i=1}^m x_i \xi_i\) satisfies
    \[
    \mathcal{L}(S,\varepsilon) \leq p.
    \]
\end{lemma}

\section{Compressible singular vectors} \label{compressible section}

We begin by proving a lemma (Lemma \ref{compcomp}) which will enable us to prove that with probability \(1 - 2e^{-c_{\ref{nocompressible}}n}\), all singular vectors of our matrix 
model are incompressible, where we introduce the relevant definitions and constants below. 

\begin{definition}
    Let \(c_0 > 0\). A vector \(v = (v_1,\dots v_m) \in \mathbb{R}^{m}\) is called \textit{\(c_0\)-sparse} if it is supported on at most \(\lfloor c_0 m \rfloor \) coordinates (where \(i \in \left\{1, 2, ..., m\right\}\) is in the support of \(v\) if \(v_i \neq 0\)).
\end{definition}

\begin{definition}
    Let \(c_0,c_1 > 0\). A unit vector \(v \in S^{m-1}\) is known as \( (c_{0}, c_{1}) \)-\textit{compressible} (or just \textit{compressible} if one is suppressing parameters) if it is within distance \(c_1\) of the set of \(c_0\)-sparse vectors. Such a set is denoted by \(\comp_m(c_0,c_1)\). In other words,
    \[
    \comp_m(c_0,c_1) = \set{ v \in S^{m-1} \colon \exists c_0\text{-sparse } x \text{ such that } \norm{v-x} \leq c_1}.
    \]
    A unit vector which is not compressible is called \textit{incompressible}. The set of incompressible vectors is denoted as \(\incomp_m(c_0,c_1) := S^{m-1}\setminus \comp_{m}(c_0,c_1)\).
\end{definition}

\begin{lemma}\label{compcomp} Let \(M\) and \(\Sigma\) be as in Assumptions \ref{matrix_model} and \ref{opnorm-assumptions}. There exist positive constants \(c_0, c_1, c_{\ref{compcomp}}\) depending on \(c_{\ref{matrix_model}},K,m_4\) and \(L\) such that
\begin{equation}\label{compeq-n}
\sup_{0 \leq \sigma^2 \leq K^2 L^2 n} \PP \set{ \inf_{v \in \comp_p(c_0,c_1)} \norm{M^{\mathrm{T}} \Sigma M v - \sigma^2 v} \leq c_{\ref{compcomp}} n \wedge \mathcal{E}_K } \leq  2 e^{-c_{\ref{compcomp}} n} 
\end{equation}
and 
\begin{align}\label{compeq-p}
   \sup_{0 \leq \sigma^2 \leq K^2 L^2 n} \PP \set{ \inf_{\Sigma^{1 / 2}u/ \norm{\Sigma^{1 / 2}u} \in \comp_n(c_0,c_1)} \norm{\Sigma^{1 / 2}M M^{\mathrm{T}} \Sigma^{1 / 2} u - \sigma^2 u} \leq c_{\ref{compcomp}} n \wedge \mathcal{E}_K } \\
\nonumber \leq 2 e^{-c_{\ref{compcomp}} n}.
\end{align}
\end{lemma}

\begin{proof}
Note that we will introduce \(c_0,c_1, c_{\ref{compcomp}}\) as arbitrary small constants, before fixing them at the end of the proof, to ensure that the result of the lemma holds. 

Our proof strategy is analogous to that appearing in the proof of \cite[Proposition 4.2]{VershyninInvert}.

Let \(\sigma^2 \in [0,K^2L^2n]\) remain fixed throughout. Notice that for a unit vector \(u\),
\begin{equation}
\nonumber \left\|\Sigma^{1 / 2}M M^{\mathrm{T}}\Sigma^{1 / 2}u - \sigma^2 u\right\| \leq c_{\ref{compcomp}}n
\end{equation}
implies 
\begin{equation}\label{eq:comp_u}
\left\|M M^{\mathrm{T}}w - \sigma^2 \Sigma^{-1} w\right\| \leq L^2 c_{\ref{compcomp}}n,
\end{equation}
\begin{details}
    \begin{align*}
        \left\|\Sigma^{1 / 2}M M^{\mathrm{T}}\Sigma^{1 / 2}u - \sigma^2 u\right\| 
        &= \left\|\left\|\Sigma^{1 / 2}u\right\|\Sigma^{1 / 2} \left( M M^{\mathrm{T}} w - \sigma^2 \Sigma^{-1} w \right)\right\| \\
        &\geq \sigma_{\text{min}}(\Sigma^{1 / 2})^2 \left\|M M^{\mathrm{T}}w - \sigma^2 \Sigma^{-1} w\right\|\\
        &\geq L^{-2}\left\|M M^{\mathrm{T}}w - \sigma^2 \Sigma^{-1} w\right\|,
    \end{align*}
\end{details}
where \(w = \Sigma^{1 / 2}u / \norm{\Sigma^{1 / 2}u}\). So we will consider \eqref{eq:comp_u} in place of its analogous expression in \eqref{compeq-p}.
By \cite[Lemma 3.7]{VershyninInvert}, there exists a \((12c_1)\)-net \(\mathcal{N}\) of \(\comp_n(c_0,c_1)\) such that
\[
\abs{\mathcal{N}} \leq \left( \frac{9}{c_0 c_1} \right)^{c_0 n},
\]
which also serves as an upper bound on the cardinality of a \((12 c_1)\)-net \(\mathcal{M}\) of \(\comp_p(c_0,c_1)\), as \(p \leq n\). Furthermore, an inspection of the proof in \cite{VershyninInvert} shows that we can assume these nets consist only of sparse unit vectors.

For \(v \in \comp_{p}(c_0,c_1)\) and \(w \in \comp_n(c_0,c_1)\), one can find sparse unit vectors \(v_0 \in \mathcal{M}\), \(w_0 \in \mathcal{N}\) such that 
\[
\norm{v - v_0} \leq 12c_1 \quad \text{and} \quad \norm{w - w_0} \leq 12c_1.
\]
The event 
\[
\norm{M^{\mathrm{T}} \Sigma M v - \sigma^2 v} \leq c_{\ref{compcomp}} n \wedge \mathcal{E}_K
\]
implies that
\begin{align}\label{compeq-v0}
\norm{M^{\mathrm{T}} \Sigma M v_0 - \sigma^2 v_0} &\leq \norm{M^{\mathrm{T}} \Sigma M (v_0 - v) - \sigma^2(v_0-v)} + \norm{M^{\mathrm{T}} \Sigma Mv - \sigma^2 v} \\
\nonumber        & \leq 12c_1(K^2 L^2 n + \sigma^2) + c_{\ref{compcomp}}n \\             \nonumber        &\leq (c_{\ref{compcomp}} + 24c_{1}K^2 L^2)n.
\end{align}
Therefore,
\begin{align*}
\PP &\set{ \inf_{v \in \comp_p(c_0,c_1)} \norm{M^{\mathrm{T}} \Sigma M v - \sigma^2 v} \leq c_{\ref{compcomp}} n \wedge \mathcal{E}_K } \\
&\qquad \qquad \leq \PP \set{ \inf_{v_0 \in \mathcal{M}} \norm{M^{\mathrm{T}} \Sigma M v_0 - \sigma^2 v_0} \leq (c_{\ref{compcomp}} + 24c_{1}K^2 L^2)n }. 
\end{align*}

Similarly, the event 
\[
\norm{\Sigma^{1 / 2}M M^{\mathrm{T}} \Sigma^{1 / 2} u - \sigma^2 u} \leq c_{\ref{compcomp}} n \wedge \mathcal{E}_K
\]
implies that
\begin{align}\label{compeq-w0}
&\norm{M M^{\mathrm{T}} w_0 - \sigma^2 \Sigma^{-1} w_0} \\
&\leq \left\|MM^{\mathrm{T}} (w_0 - w) - \sigma^2\Sigma^{-1}(w_0 - w)\right\| + \left\|MM^{\mathrm{T}} w - \sigma^2 \Sigma^{-1} w\right\|\\
\nonumber &\leq 12 c_1 (K^2 n  +\sigma^2 L^2) + L^2 c_{\ref{compcomp}} n \\
\nonumber &\leq (L^2 c_{\ref{compcomp}} + 24c_{1} K^2 L^4 )n,
\end{align}
where we have used that both \(L,K \geq 1\). 
Thus, 
\begin{align*}
    \PP &\set{ \inf_{\Sigma^{1 / 2}u/ \norm{\Sigma^{1 / 2}u} \in \comp_n(c_0,c_1)} \norm{\Sigma^{1 / 2}M M^{\mathrm{T}} \Sigma^{1 / 2} u - \sigma^2 u}} \\
    &\qquad \leq \PP \set{ \inf_{w_0 \in \mathcal{N}} \norm{M M^{\mathrm{T}} w_0 - \sigma^2 \Sigma^{-1} w_0} \leq (L^2 c_{\ref{compcomp}} + 24c_{1} K^2 L^4 )n}.
\end{align*}

We now fix $v_0 \in \mathcal{M}$ and $w_0 \in \mathcal{N}$.  By reordering if necessary, we will assume that \(v_0\) and \(w_0\) are supported in the first \(\left\lfloor c_0 p\right\rfloor\) and \(\left\lfloor c_0 n\right\rfloor\) coordinates, respectively.
We first consider equation \eqref{compeq-v0}. We decompose \(M\) as:
\begin{equation} \label{eq:decomp 1}
M=
\begin{pmatrix}
X \rvert Y
\end{pmatrix},
\end{equation}
where \(X\) is an \(n \times \lfloor c_0p \rfloor\) matrix and \(Y\) is an \( n \times \lceil(1-c_0)p \rceil \) matrix. Note that the prior assumption that \(v_0\) is supported on its first \(\lfloor c_0 p \rfloor\) coordinates does not incur any entropy cost, as it only for notational convenience when we write our decomposition \eqref{eq:decomp 1}. We simply make this assumption so that subsequent equations involving this decomposition are cleaner and easier to follow. 

It then follows that
\begin{equation}
\nonumber M^{\mathrm{T}} \Sigma M v_0 - \sigma^2 v_0 = 
\left(\begin{array}{@{}c|c@{}}
    X^{\mathrm{T}} \Sigma X & X^{\mathrm{T}} \Sigma Y \\ \hline
    Y^{\mathrm{T}}\Sigma X & Y^{\mathrm{T}}\Sigma Y
  \end{array}\right)  
  \left(
  \begin{array}{c}
  v_{0 1} \\
  \vdots \\
  v_{0 \left\lfloor c_0 p\right\rfloor} \\
  0 \\
  \vdots \\
  0 
  \end{array} \right)
   - \sigma^2 \left(
  \begin{array}{c}
  v_{01} \\
  \vdots \\
  v_{0 \left\lfloor c_0 p\right\rfloor} \\
  0 \\
  \vdots \\
  0 
  \end{array} \right).
\end{equation}
Looking at the last \(\lceil (1-c_0)p \rceil\) coordinates of the right-hand side of this equation gives:
\[
\norm{M^{\mathrm{T}} \Sigma M v_0 - \sigma^2 v_0} \geq \norm{Y^{\mathrm{T}} \Sigma X \left(
  \begin{array}{c}
  v_{01} \\
  \vdots \\
  v_{0 \left\lfloor c_0 p\right\rfloor} \\
  \end{array} \right)}.
  \]
So that if \(\norm{M^{\mathrm{T}} \Sigma M v - \sigma^2 v} \leq c_{\ref{compcomp}} n\), then
\begin{equation}
\norm{Y^{\mathrm{T}} \Sigma X \left(
  \begin{array}{c}
  v_{01} \\
  \vdots \\
  v_{0 \left\lfloor c_0 p\right\rfloor} \\
  \end{array} \right)}
  \leq 
  (c_{\ref{compcomp}} + 24c_{1}K^2 L^2)n. \label{eq:3.5}
\end{equation}

For notational convenience, we shall denote
\[
v_0' := \left(
  \begin{array}{c}
  v_{01} \\
  \vdots \\
  v_{0 \left\lfloor c_0 p\right\rfloor} \\
  \end{array} \right),
\]
which is a unit vector, due to being the projection of the unit vector \(v_0\) onto its support. 
Let us apply the Tensorization Lemma (Lemma \ref{tensorization}) to both the rows of \(Xv_0'\) and then to the rows of \(Y^{\mathrm{T}}(\Sigma Xv_0')\). Note that the \(k\)th entry of \(Xv_0'\) is equal to the sum
\begin{equation}\label{compmeansum}
    (Xv_0')_{k} = \sum_{j=1}^{c_0 p} v_{0j} \xi_{kj}.
\end{equation}

As each \(\xi_{kj}\) is independent and identically distributed, each \((Xv_0')_{k}\) is independent and identically distributed. So by Lemma \ref{sum anticonc}:
\begin{equation}\label{compmeanshift}
    \mathcal{L}((Xv_0')_k, 1 / 2) \leq \tilde{c}
\end{equation}
for some constant \(\tilde{c}\) depending on \(m_4\). 

Then the Tensorization Lemma can be applied over the \(n\) rows of \(X\) to give:
\[
\mathcal{L}(Xv_0' , \varepsilon \sqrt{n}) \leq e^{-c_{\ref{compcomp}}' n}
\]
for some constants \(\varepsilon, c_{\ref{compcomp}}' \) depending only on \(\tilde{c}\). By taking \(\hat{c}_{\ref{compcomp}} = \min \left\{ c_{\ref{compcomp}}', \varepsilon \right\}\), and looking at concentration around \(0\), this tells us that
\[
\PP(\norm{Xv_0'} \leq \hat{c}_{\ref{compcomp}} \sqrt{n} ) \leq e^{-\hat{c}_{\ref{compcomp}}n}.
\]

Now we can condition on \(Xv_0'\), and assume for now that \(\norm{Xv_0'} \geq  \hat{c}_{\ref{compcomp}} \sqrt{n}\), so that \(\norm{\Sigma X v_0'} \geq L^{-2} \hat{c}_{\ref{compcomp}} \sqrt{n}\). Now, \(\Sigma Xv_0'/ \norm{\Sigma Xv_0'}\) is a unit vector, and the same argument yields
\begin{equation*}
\PP \left(\norm{Y^{\mathrm{T}} \left(\Sigma Xv_0'/\norm{\Sigma Xv_0'}\right)} \leq \hat{c}_{\ref{compcomp}} \sqrt{ \lceil(1-c_0)p \rceil} \right) \leq e^{-\hat{c}_{\ref{compcomp}} \lceil (1-c_0) p \rceil} 
\end{equation*}
(the factor of \( \lceil (1-c_0) p \rceil\) comes from the number of rows of \(Y^{\mathrm{T}}\)). Hence the following holds with probability at least \(1-( e^{-\hat{c}_{\ref{compcomp}} n} + e^{-\hat{c}_{\ref{compcomp}}(1-c_0) p})\):
\begin{align*}\label{smallsparsenorm-n}
   \norm{M^{\mathrm{T}} \Sigma M v_0 - \sigma^2 v_0} \geq \norm{Y^{\mathrm{T}} \Sigma Xv_0'} 
    &\geq (\hat{c}_{\ref{compcomp}} \sqrt{\lceil(1-c_0)p \rceil})( L^{-2} \hat{c}_{\ref{compcomp}} \sqrt{n}) \\
  \nonumber  &\geq \left(  L^{-2} \hat{c}_{\ref{compcomp}}^2 \sqrt{(1-c_0)} \sqrt{\lambda} \right) n \\
  \nonumber  &\geq \left( L^{-2} \hat{c}_{\ref{compcomp}}^2 \sqrt{1-c_0} \sqrt{c_{\ref{matrix_model}}} \right) n.
\end{align*}
We choose constants \(c_1, c_{\ref{compcomp}}\) sufficiently small so that 
\begin{equation}\label{eq:choose_bad_comp_constants}
    \left( L^{-2} \hat{c}_{\ref{compcomp}}^2 \sqrt{(1-c_0)} \sqrt{c_{\ref{matrix_model}}} \right)n \geq (c_{\ref{compcomp}} + 24c_1K^2L^2)n.
\end{equation}

Let us simplify things and observe that
\begin{equation*}
e^{-\hat{c}_{\ref{compcomp}} n} + e^{-\hat{c}_{\ref{compcomp}}(1-c_0) p} = e^{-\hat{c}_{\ref{compcomp}} n} + e^{-\hat{c}_{\ref{compcomp}}(1-c_0) \lambda n} \leq 2 e^{- \hat{c}_{\ref{compcomp}} (1 - c_0)\lambda n} \leq 2  e^{- \hat{c}_{\ref{compcomp}} c_{\ref{matrix_model}} n / 2},
\end{equation*}
so long as \(c_0 \leq 1 / 2\). 

Now let us consider equation \eqref{compeq-w0}. We decompose \(M\) as
\begin{equation*}
M=
\begin{pmatrix}
X \\ \hline Y
\end{pmatrix},
\end{equation*}
where \(X\) is a \(\lfloor c_0 n\rfloor \times p\) matrix and \(Y\) is a \( \lceil (1-c_0)n \rceil \times p\) matrix. Thus, we have that
\begin{equation}
\nonumber M M^{\mathrm{T}} w_0 - \sigma^2 \Sigma^{-1} w_0 = 
\left(\begin{array}{@{}c|c@{}}
    X X^{\mathrm{T}} & X Y^{\mathrm{T}} \\ \hline
    Y X^{\mathrm{T}} & Y Y^{\mathrm{T}}
  \end{array}\right)  
  \left(
  \begin{array}{c}
  w_{0 1} \\
  \vdots \\
  w_{0 \left\lfloor c_0 n\right\rfloor} \\
  0 \\
  \vdots \\
  0 
  \end{array} \right)
   - \sigma^2 \Sigma^{-1} \left(
  \begin{array}{c}
  w_{01} \\
  \vdots \\
  w_{0 \left\lfloor c_0 n\right\rfloor} \\
  0 \\
  \vdots \\
  0 
  \end{array} \right).
\end{equation}
If we look at the final \(\lceil(1-c_0)n\rceil\) coordinates of this equation, we see that
\begin{align*}
    \norm{M M^{\mathrm{T}} w_0 - \sigma^2 \Sigma^{-1}w_0} \geq \norm{Y X^{\mathrm{T}} w_0' - \sigma^2 P^{\perp} \Sigma^{-1} w_0},
\end{align*}
where \(P^{\perp}\) is the projection onto the final \(\left\lceil (1-c_0)n\right\rceil\) coordinates of \(\mathbb{R}^n\) and 
\[
w_0' := \left(
  \begin{array}{c}
  w_{01} \\
  \vdots \\
  w_{\left\lfloor c_0 n\right\rfloor} \\
  \end{array} \right).
\]
Note that \(w_0'\) is still a unit vector, as it is the projection of \(w_0\) onto its support. 

By the same argument as above, bounds on \(\mathcal{L}(X^{\mathrm{T}}w_0', \varepsilon)\) and the Tensorization Lemma (Lemma \ref{tensorization}) tell us that
\begin{equation*}
   \PP( \norm{X^{\mathrm{T}} w_0'} \leq \hat{c}_{\ref{compcomp}} \sqrt{p} ) \leq e^{- \hat{c}_{\ref{compcomp}} p}
\end{equation*}
for some constant \(\hat{\alpha}_{\ref{compcomp}}\) depending on \(m_4\).

Now, let us condition on the entries of \(X^{\mathrm{T}}\). The anti-concentration of
\[
   Y \left( X^{\mathrm{T}} w_0' / \norm{X^{\mathrm{T}} w_0'} \right)
\]
around the point \( \sigma^2 P^{\perp} \Sigma^{-1} w_0/ \norm{X^{\mathrm{T}} w_0'}\) yields that with probability at least \(1 - (e^{-\hat{c}_{\ref{compcomp}}p} + e^{-\hat{c}_{\ref{compcomp}} \lceil(1-c_0)\rceil n}) \geq 1 - 2 e^{-\hat{c}_{\ref{compcomp}} c_{\ref{matrix_model}} n / 2}\), we have that
\begin{align*}\label{smallsparsenorm-p} 
\norm{M M^{\mathrm{T}} u_0 - \sigma^2 \Sigma^{-1} w_0} &\geq \norm{YX^{\mathrm{T}} w_0' - \sigma^2 P^{\perp} \Sigma^{-1} w_0} \\  
\nonumber &\geq \left(\hat{c}_{\ref{compcomp}} \sqrt{ \lceil(1-c_0)p \rceil } \right)( \hat{c}_{\ref{compcomp}} \sqrt{n}) \\
\nonumber &\geq \left( \hat{c}_{\ref{compcomp}}^2 \sqrt{1-c_0} \sqrt{c_{\ref{matrix_model}}} \right) n.
\end{align*}

We may then choose \(c_1, c_{\ref{compcomp}}\) so that \eqref{eq:choose_bad_comp_constants} and the following inequality,
\begin{equation*}\label{eq:choose_comp_constants} 
    \left( \hat{c}_{\ref{compcomp}}^2 \sqrt{1-c_0} \sqrt{c_{\ref{matrix_model}}} \right) n \geq (L^2 c_{\ref{compcomp}} + 24c_{1} K^2 L^4 )n
\end{equation*}
both hold. With this choice of \(c_1, c_{\ref{compcomp}}\), we have that
\begin{align}
    \mathbb{P}\left( \left\|M^{\mathrm{T}}\Sigma M v_0 - \sigma^2 v_0\right\| \leq (c_{\ref{compcomp}} + 24 c_1 K^2L^2)n \wedge \mathcal{E}_K \right) \leq 2\exp(-\hat{c}_{\ref{compcomp}}c_{\ref{matrix_model}}n / 2)\label{eq:comp_prob_v_no_net},
\end{align}
and
\begin{align}
    \mathbb{P}\left( \left\|M M^{\mathrm{T}} w_0 - \sigma^2 \Sigma^{-1} w_0\right\| \leq (L^2c_{\ref{compcomp}} + 24 c_1 K^2L^{4})n \wedge \mathcal{E}_K \right) \leq 2\exp(-\hat{c}_{\ref{compcomp}}c_{\ref{matrix_model}}n / 2).   \label{eq:comp_prob_u_no_net}
\end{align}

Now, we union bound over \(\mathcal{M}\) or \(\mathcal{N}\) as appropriate. Let \(\mathcal{E}\) denote either the event in equation \eqref{compeq-n} or equation \eqref{compeq-p}. The event \(\mathcal{E}\) implies the event in equation \eqref{eq:comp_prob_v_no_net} or equation \eqref{eq:comp_prob_u_no_net}, for some \(v_0 \in \mathcal{M}\) or \(u_0 \in \mathcal{N}\). Using that \(\abs{\mathcal{N}},\abs{\mathcal{M}} \leq \left( \frac{9}{c_0 c_1} \right)^{c_0 n} \), we observe that
\begin{align*}
    \PP(\mathcal{E}) &\leq \left( \frac{9}{c_0 c_1} \right)^{c_0 n} \left( 2 e^{-\hat{c}_{\ref{compcomp}}c_{\ref{matrix_model}} n / 2} \right) \\
&\leq 2 e^{\hat{c}_{\ref{compcomp}}c_{\ref{matrix_model}} n / 4} e^{-\hat{c}_{\ref{compcomp}}c_{\ref{matrix_model}} n / 2} = 2 e^{-\hat{c}_{\ref{compcomp}}c_{\ref{matrix_model}} n / 4},
\end{align*}
by choosing \(c_0\) sufficiently small so that \((9/c_0c_1)^{c_0} \leq e^{\hat{c}_{\ref{compcomp}}c_{\ref{matrix_model}} / 4} \), concluding the proof.
\end{proof}

\begin{remark}\label{rem:mean non zero}
   Let us note on the modifications to this argument if \(M\) has entries with non-zero mean, which is a rank-1 perturbation. Write \(M = \Xi + \mu\), where \(\Xi\) is as in Assumption \ref{matrix_model} and \(\mu\) is the matrix of means. Note that for our results to be meaningful, \( \mathbb{P}(\mathcal{E}_K^{c}) \) must be small, for which we require \( \|\mu\| = O(\sqrt{n}) \). Then equation \eqref{compmeansum} reads as
\begin{equation*}
    (Xv_0')_{k} = \sum_{j=1}^{c_0 p} v_{0j} \xi_{kj} + \sum_{j=1}^{c_0 p} v_{0j} \mu_{kj}.
\end{equation*}
It then follows that
\begin{align*}
    \mathcal{L}((Xv_0')_k, 1 / 2) &= \sup_{a \in \RR} \mathbb{P}\set{ \norm{(Xv_0')_k - a} \leq 1/2} \\
    & \leq \sup_{a \in \RR} \mathbb{P}\set{ \norm{ \sum_{j=1}^{c_0 p} v_{0j} \xi_{kj} - \left(a - \sum_{j=1}^{c_0 p} v_{0j} \mu_{kj}\right)} \leq 1/2} \\
    &\leq \mathcal{L}\left(\sum_{j=1}^{c_0 p} v_{0j} \xi_{kj}, 1 / 2 \right),
\end{align*}
as the supremum over \(a\) is equivalent to taking the supremum over \((a - \sum_{j=1}^{c_0 p} v_{0j} \mu_{kj})\). Thus, in the case that \(\mu \neq 0\), we can bound all small ball probabilities by the corresponding small ball probability when \(\mu = 0\), and so the argument proceeds as in the \(\mu=0\) case. All subsequent arguments involving small ball probabilities throughout the paper (lines \eqref{incompmeansumB}, \eqref{incompmeansumC} and \eqref{eq:lastcolumnmeansum} and their subsequent lines \eqref{incompmeanshiftB}, \eqref{incompmeanshiftC} and \eqref{eq:lastcolumnmeanshift}) proceed in an identical way to the \(\mu \neq 0\) case.
\end{remark}
For the remainder of the paper, whenever \(c_0\) and \(c_1\) appear they will be treated as fixed constants, chosen so that Lemma \ref{compcomp} holds.
Now, by applying a strategy similar to \cite[Lemma 5.3]{nguyentaovugaps}, that is, by approximating potential (squared) singular values using a multiple of \(c_{\ref{compcomp}} n\) in the interval \([0,L^2K^2n]\), we use Lemma \ref{compcomp} to prove the following:

\begin{proposition}\label{nocompressible}
    Let \(M\) and \(\Sigma\) be as in Assumptions \ref{matrix_model} and \ref{opnorm-assumptions}. Let \(\mathcal{G}\) be the event that there exists a positive \(\sigma^2\) and  unit vectors \(v\) and \(u\) such that
    \begin{equation*}\label{sveqs} 
       \left(  M^{\mathrm{T}}\Sigma Mv = \sigma^2 v \right) \text{ and } \left( \Sigma^{1 / 2}M M^{\mathrm{T}}\Sigma^{1 / 2}u = \sigma^2 u \right)
    \end{equation*}
    and either \(v \in \comp_p(c_0,c_1)\) or \(\Sigma^{1/2} u/ \norm{\Sigma^{1/2}u} \in \comp_n(c_0,c_1)\). 
    Then 
    \begin{equation*}
        \mathbb{P}(\mathcal{G} \wedge \mathcal{E}_K) \leq 2 e^{-c_{\ref{nocompressible}} n}
    \end{equation*}
    for some constant \(c_{\ref{nocompressible}}\) which depends on \(c_0, c_1, c_{\ref{matrix_model}},K,m_4\) and \(L\).
\end{proposition}

\begin{proof}
We use a net argument to reduce the argument to fixed $\sigma^2$.  Assume that there is a compressible right singular vector \(v\) with singular value \(\sigma^2\). Also assume that \(\mathcal{E}_K\) applies, so that \(\norm{M^{\mathrm{T}}\Sigma M}, \norm{\Sigma^{1 / 2}M M^{\mathrm{T}}\Sigma^{1 / 2}} \leq L^2K^2 n \). Let \(\mathcal{N}\) be the set of non-zero multiples of \(c_{\ref{compcomp}}n\) in the interval \([0, L^2K^2 n]\) (so that \(|\mathcal{N}| \leq \left\lfloor L^2K^2/c_{\ref{compcomp}}\right\rfloor\) ). The event \(\mathcal{G} \wedge \mathcal{E}_K\) implies that there exists \(\sigma_0^2 \in \mathcal{N}\) such that
\begin{align}\label{approxSV}
    \norm{M^{\mathrm{T}}\Sigma M v - \sigma_0^2 v} &= \abs{\sigma^2 - \sigma_0^2} \norm{v} \leq c_{\ref{compcomp}} n.
\end{align}
By Lemma \ref{compcomp}, this occurs (for any such \(v\)) with probability less than or equal to \( 2 e^{-c_{\ref{compcomp}}n}\). Union bounding over all possible values of \(\sigma_0^2 \in \mathcal{N}\) gives
 \begin{align*}
\PP &\set{ \exists \text{ a right singular vector \(v\) of \(\Sigma^{1/2} M\) such that \(v \in \comp_p(c_0,c_1)\)}\wedge \mathcal{E}_K}\\
&\leq \frac{L^2K^2}{c_{\ref{compcomp}}} 2 e^{-c_{\ref{compcomp}} n}.
\end{align*}
A similar calculation yields the other inequality,
\begin{align*}
\PP &\set{ \exists \text{ a left singular vector \(u\) of \(\Sigma^{1/2}M\) s.t. \(\frac{\Sigma^{1/2}u}{\left\|\Sigma^{1/2}u \right\|} \in \comp_{n}(c_0,c_1)\)}\wedge \mathcal{E}_K} \\
    &\leq \frac{L^2K^2}{c_{\ref{compcomp}}} 2 e^{-c_{\ref{compcomp}} n}.
\end{align*}
Summing these two bounds and choosing \(c_{\ref{nocompressible}}\) sufficiently small so that 
\[
\min \left\{ \frac{4L^2K^2}{c_{\ref{compcomp}}} e^{-c_{\ref{compcomp}} n}, 1 \right\} \leq 2e^{-c_{\ref{nocompressible}}n} 
\]
for all \(n\) yields the result. In future proofs, we will perform this final step of replacing any leading constant in probability bounds with the constant \(2\) without comment.
\end{proof}

\section{Incompressible singular vectors and regularized LCD} \label{incompressible section}

Having shown that there is only a small probability that any singular vector is compressible, in this section we turn our attention to the incompressible vectors. Our ultimate goal is to show that singular vectors are sufficiently disordered, so that when their dot product is taken with a column of our random matrix \(M\) the resulting quantity has strong anti-concentration properties. We then use this to give a lower bound on the gap size between squared singular values.

Compressible vectors give relatively weak anti-concentration bounds. However, this may also be the case for an incompressible vector such that its coordinates are close to having additive structure. For example, the vector \(\frac{1}{\sqrt{n}}(1,\dots,1)\) has relatively weak anti-concentration, due to the potential of many cancellations occurring if we were, for example, to take its dot product with a Rademacher random vector. Therefore, we wish to show that a singular vector is unlikely to have such structure. The key tool for detecting the extent to which a particular vector has strong anti-concentration bounds is the \textit{least common denominator}, or \(LCD\). In this case, we use the so-called \textit{regularized LCD}, which is more robust under the types of decomposition arguments used during our proofs. At the end of this section, we use a somewhat technical \(\varepsilon\)-net argument to demonstrate that in general, singular vectors have sufficiently large regularized LCD with high probability. 
  
\subsection{Small ball probabilities via LCD} 

We use the definition of \textit{least common denominator} (LCD) as in \cite{nguyentaovugaps}, a concept originally introduced by Rudelson and Vershynin (see \cite{rudelson2008littlewoodoffordprobleminvertibilityrandom}, \cite{rudelson2009smallestsingularvaluerandom}). 
We begin by including the relevant concepts and notation from \cite{nguyentaovugaps} needed for our argument. Given parameters \(\kappa,\gamma\) with \(\gamma \in (0,1)\) and \(\kappa>0\), for a unit vector \(x \in S^{m-1} \):
\[
\LCD_{\kappa,\gamma}(x) := \inf \left\{ \theta > 0 \colon \text{dist}(\theta x, \mathbb{Z}^m) < \min \left( \gamma \norm{\theta x}, \kappa \right) \right\},
\]
where \(\kappa\) and \(\gamma\) can potentially depend on \(m\). We primarily utilize the related notion of \textit{regularized LCD} (originally introduced by Vershynin in \cite{VershyninInvert}), which is defined in \cite{nguyentaovugaps} as
\[
\widehat{\LCD}_{\kappa,\gamma}(x,\alpha) = \max \left\{ \LCD_{\kappa,\gamma} (x_{I}/\norm{x_{I}}) \colon I \subset \text{spread}(x), \abs{I} = \ceil{\alpha n} \right\} ,
\]
where \(x \in \incomp_m(c_0,c_1)\), and \(\text{spread}(x)\) is a subset of the indices of \(\{1, \ldots, m\}\) such that 
\begin{equation}\label{eq:well spread}
\frac{c_1}{\sqrt{2m}} \leq \abs{x_k} \leq \frac{1}{\sqrt{c_0 m}}
\end{equation}
for every \(k \in \text{spread}(x)\), and 
\begin{equation} \label{eq:c'} \abs{\text{spread}(x)} = \ceil{c'm} \text{ where } c' = c_0 c_1^2/4. \end{equation} Note that such a subset \(\text{spread}(x)\) of the specified cardinality exists by the definition of incompressibility (for a proof, see \cite{rudelson2008littlewoodoffordprobleminvertibilityrandom} Lemma 3.4).
The parameter \(\alpha\) satisfies \(0 < \alpha < c'/4\), where \(\alpha\) may be chosen to depend on \(m\).

Note that for the purpose of proving our main results we will only need the case where \(\gamma\) in the definition of LCD is a constant, which we choose to be \(1/2\), and we shall set \(\kappa = n^{2c}\), where \(c\) is a small constant appearing in Theorem \ref{thm:main}. All constants in our results will be independent of both \(\gamma\) and \(\kappa\). 

The importance of regularized LCD is its relation to small ball probabilities.

\begin{lemma}[\cite{nguyentaovugaps} Lemma 5.8] \label{incompanticonc}
Assume that \(\xi_1, \dots \xi_m\) are i.i.d. (real-valued) random variables and let \(\xi = (\xi_1,\dots , \xi_m)\). Assume also that there exist \(\varepsilon_0, p_0 > 0\) such that \(\mathcal{L}(\xi_1,\varepsilon_0) \leq 1 - p_0\). If \(x \in \incomp_m(c_0,c_1)\), then for \(\kappa > 0\) and any \(\varepsilon\) which satisfies
\begin{equation*}\label{eq:anti-concentration_single} 
\varepsilon \geq \frac{\sqrt{\alpha}}{c_0 \widehat{\LCD}_{\kappa,\gamma}(x,\alpha)}
\end{equation*}
there exists a constant \(C_{\ref{incompanticonc}} > 0\) depending only on \(\varepsilon_0\) and \(p_{0}\) such that
\begin{equation*}\label{regularizedsmallball}
\mathcal{L}(\xi \cdot x,\varepsilon) \leq C_{\ref{incompanticonc}} \left( \frac{\varepsilon}{\gamma c_1 \sqrt{\alpha}} + e^{- \Omega(\kappa^2)} \right).
\end{equation*}
\end{lemma}

\begin{remark}
    While \cite[Lemma 5.8]{nguyentaovugaps} is stated for subgaussian \( \xi_{i} \), by replacing \cite[Lemma 5.5]{nguyentaovugaps} in the proof of \cite[Lemma 5.8]{nguyentaovugaps} with \cite[Theorem 3.4]{rudelson2009smallestsingularvaluerandom}, we can replace the subgaussian assumption with the uniform anti-concentration assumption \( \mathcal{L}( \xi_{i}, \varepsilon_{0}) \leq 1-p_{0} \).
\end{remark}
\begin{remark}
       For the random variables in Lemma \ref{PZ}, \(\varepsilon_0, p_0\) exist by the assumption that \(\Var \xi = 1, m_4 < \infty\). Furthermore, they can be chosen to only depend on \(m_4\). 
\end{remark}

\begin{comment} 
\begin{remark}
In fact, the proof of the above lemma shows that the result holds for \(x\) replaced by \(x_{J}\), where \(J\) is any subset of coordinates that contains \(I\), the set of coordinates where the regularized \( \LCD \) is achieved.
\end{remark}
\end{comment}

We use the above lemma to give a bound on the simultaneous anti-concentration of \(\Sigma^{1 / 2}M v \) and \(M^{\mathrm{T}} \Sigma^{1 / 2} u\). In the case of incompressible singular vectors, it is more convenient to work with the intersection of events 
\[\set{\Sigma^{1/2}M v = \sigma u} \wedge \set{M^{\mathrm{T}} \Sigma^{1/2} u  = \sigma v}\]
(i.e. the singular vector equations for \(\Sigma^{1/2}M\)) instead of the individual events
\[\set{M^{\mathrm{T}} \Sigma M v = \sigma^2 v}
\]
or
\[\set{\Sigma^{1/2} M M^{\mathrm{T}} \Sigma^{1/2} u = \sigma^2 u},\]
(i.e., the eigenvalue equations for \(M^{\mathrm{T}} \Sigma M\) and for \(\Sigma^{1/2} M M^{\mathrm{T}} \Sigma^{1/2}\)), as the required anti-concentration results are difficult to apply to matrices of the form \(M^{\mathrm{T}} \Sigma M\) due to the lack of independent rows and columns. 

The next proposition is stated without reference to \(\Sigma\). It bounds the probability that any fixed pair \((x,y)\) of incompressible vectors are \textit{approximate} singular vectors, with the level of approximation dependent on the regularized LCD of \(x\) and \(y\). After proving this proposition, we employ a net argument in the next section, at which point we manipulate the inequalities appearing in the event in \(\eqref{anticoncresult}\) into a form that accounts for the covariance matrix \(\Sigma\).

\begin{lemma}\label{anti-concentration argument}
Let the matrix \(M\) satisfy Assumption \ref{matrix_model}. 
Then there exists a constant \(C_{\ref{anti-concentration argument}}\) (depending only on \(m_4\)) such that for any \(x \in \incomp_p(c_0,c_1)\), \(y \in \incomp_n(c_0,c_1)\), \(0 <\alpha < c'/4 \) (where \(c'\) was defined in \eqref{eq:c'}), and \(\varepsilon, \varepsilon' > 0 \) satisfying
\begin{equation}\label{eq:epsilon geq LCD}
\varepsilon \geq \frac{\sqrt{\alpha}}{c_0 \widehat{\LCD}_{\kappa,\gamma}(x,\alpha)}, \qquad \varepsilon' \geq \frac{\sqrt{\alpha}}{c_0 \widehat{\LCD}_{\kappa,\gamma}(y,\alpha)},
\end{equation}
we have the following:
\begin{align}
\begin{split}
    \sup_{(a,b) \in \RR^n \times \RR^p} &\PP(\set{\norm{M x - a} \leq \varepsilon \sqrt{n}} \wedge \set{\norm{M^{\mathrm{T}}y - b} \leq \varepsilon' \sqrt{p}}) \label{anticoncresult} \\
    &\leq C_{\ref{anti-concentration argument}}^{n+p} \left( \frac{\sqrt{2}\varepsilon}{\gamma c_1 \sqrt{\alpha}} + e^{- \Omega(\kappa^2)} \right)^{n - \ceil{\alpha n}} \left( \frac{\sqrt{2}\varepsilon'}{\gamma c_1 \sqrt{\alpha}} + e^{- \Omega(\kappa^2)} \right)^{p - \ceil{\alpha p}}.
\end{split}
\end{align}

\end{lemma}
\begin{proof} First, let \(I(x) \subset \set{1,\dots , p}\) be coordinates of \(x\) where \(\widehat{\LCD}_{\kappa,\gamma}(x,\alpha)\) is achieved, and \(I(y) \subset \set{1,\dots, n}\) be coordinates of \(y\) where \(\widehat{\LCD}_{\kappa,\gamma}(y,\alpha)\) is achieved.

Fix \(a \in \RR^n\), \(b \in \RR^p\). Without loss of generality (by a permutation of the rows and columns of \( M \), under which the law of \( M \) is invariant), we may assume that \( I(x) = \left\{p - \left\lceil \alpha p\right\rceil + 1, ..., p\right\}\) and that \( I(y) = \left\{n - \left\lceil \alpha n\right\rceil + 1, ..., n\right\} \). This assumption does not change our argument, it simply allows for a nice presentation of \(M\) as 
\begin{equation*}
M = \begin{pmatrix}
A & B \\
C & D 
\end{pmatrix},
\end{equation*}
where \(A\) is an \(|[n]\setminus I(y)| \times |[p] \setminus I(x)|\) matrix, \(B\) is \(|[n]\setminus I(y)| \times |I(x)|\) and \(C^{\mathrm{T}}\) is \(|[p]\setminus I(x)| \times |I(y)|\). We also write 
\begin{align*}
x = \begin{pmatrix}
x' \\ x''
\end{pmatrix},
b= \begin{pmatrix}
b' \\ b''
\end{pmatrix}, \quad \text{ and } \quad
y = \begin{pmatrix}
y' \\ y''
\end{pmatrix},
b= \begin{pmatrix}
a' \\ a''
\end{pmatrix},
\end{align*}
where \(x',b' \in \RR^{|[p]\setminus I(x)|}\), \(x'',b'' \in \RR^{|I(x)|}\) and \(y',a' \in \RR^{|[n]\setminus I(y)|}\), \(y'',a'' \in \RR^{|I(y)|}\). 
The idea of our argument is to use the independence of the blocks \(B\) and \(C\) to `factor' the event in \eqref{anticoncresult}, only losing \(\left\lceil \alpha n\right\rceil\) and \(\left\lceil \alpha p\right\rceil\) in the exponents, due to \(B\) having \(n - \lceil \alpha n \rceil\) rows and \(C^{\mathrm{T}}\) having \(p - \lceil \alpha p \rceil\) rows. 

Writing things out in terms of our decomposition, we see that
\begin{align*}
M x  - a &= \begin{pmatrix}
A & B \\
C & D
\end{pmatrix}
\begin{pmatrix}
x' \\
x''
\end{pmatrix}
-
\begin{pmatrix}
a' \\
a''
\end{pmatrix},
\\
M^{\mathrm{T}} y - b &= \begin{pmatrix}
A^{\mathrm{T}} & C^{\mathrm{T}} \\
B^{\mathrm{T}} & D^{\mathrm{T}}
\end{pmatrix}
\begin{pmatrix}
y' \\
y''
\end{pmatrix}
-
\begin{pmatrix}
b' \\
b''
\end{pmatrix}.
\end{align*}
It therefore follows that
\begin{align}
\begin{split}\label{decompnormstart}
    \norm{Mx-a}^2 
    &= \norm{Ax' + Bx'' - a'}^2 + \norm{Cx'+Dx''-a''}^2 \\
    &\geq \norm{Ax' + Bx'' - a'}^2,
\end{split}\\
\begin{split}\label{decompnormend}
    \norm{M^{\mathrm{T}} y- b}^2 
    &= \norm{A^{\mathrm{T}} y' +C^{\mathrm{T}} y'' -b'}^2 + \norm{B^{\mathrm{T}} y' + D^{\mathrm{T}} y'' - b''}^2  \\
    &\geq \norm{A^{\mathrm{T}} y' +C^{\mathrm{T}} y'' -b'}^2  .
\end{split}
\end{align}

Let us now condition on the entries of \(A\). Letting \(a_0 := a'  - A x'\) and \(b_0 := b' - A^{\mathrm{T}} y'\), which are now both fixed vectors, we see that \eqref{decompnormstart} and \eqref{decompnormend} read as
\begin{align*}
\norm{Mx - a} &\geq \norm{B x'' - a_0}, \\
\norm{M^{\mathrm{T}}y - b} &\geq \norm{C^{\mathrm{T}} y'' - b_0}.
\end{align*}

Now let \(\mathcal{B}_a = \set{\norm{B x'' - a_0} \leq \varepsilon \sqrt{n}}\) and \(\mathcal{C}_b= \set{\norm{C^{\mathrm{T}} y'' - b_0} \leq \varepsilon' \sqrt{p}}\). Let \(\mathcal{E}_{a,b}= \set{\norm{Mx  - a} \leq \varepsilon \sqrt{n}} \wedge \set{\norm{M^{\mathrm{T}} y -  b} \leq \varepsilon' \sqrt{p}}\). We have shown that
\begin{equation*}
\mathcal{E}_{a,b} \subset \mathcal{B}_a \wedge \mathcal{C}_b,
\end{equation*}
and therefore
\begin{equation*}
\PP(\mathcal{E}_{a,b}) \leq \PP(\mathcal{B}_a) \PP(\mathcal{C}_b)
\end{equation*}
by independence of \(\mathcal{B}_a\) and \(\mathcal{C}_b\), which follows from independence of the entries of \(B\) from the entries of \(C\). Note that \(\mathcal{E}_{a,b}\) is the event in \eqref{anticoncresult} with fixed \(a\) and \(b\).   

From this we deduce that 
\begin{align*}
\sup_{(a,b) \in \RR^n \times \RR^p} \PP(\mathcal{E}_{a,b}) \leq& \sup_{(a,b) \in \RR^n \times \RR^p} \PP(\mathcal{B}_a) \PP(\mathcal{C}_b) = \sup_{a \in \RR^n} \PP(\mathcal{B}_a) \sup_{b \in \RR^p} \PP(\mathcal{C}_b)\\ &= \mathcal{L}(Bx'', \varepsilon \sqrt{n}) \mathcal{L}(C^{\mathrm{T}} y'', \varepsilon' \sqrt{p}).
\end{align*}
We therefore have that for all \(i \in [n] \setminus I(y)\)
\begin{align}\label{incompmeansumB}
(B x'')_i &= \sum_{j \in I(x)} \xi_{ij} x''_j
\end{align}
and for all \(i \in [p]\setminus I(x)\)
\begin{align}\label{incompmeansumC}
(C^{\mathrm{T}} y'')_i &= \sum_{j \in I(y)} \xi_{ji} y''_j.
\end{align}

We observe that provided \(\varepsilon, \varepsilon'\) satisfy \eqref{eq:epsilon geq LCD}, so do \(\sqrt{2}\varepsilon, \sqrt{2}\varepsilon'\), and so by applying Lemma \ref{incompanticonc}, we see that
\begin{equation}\label{incompmeanshiftB}
    \mathcal{L}((B x'')_i, \sqrt{2}\varepsilon) 
    \leq C_{\ref{incompanticonc}} \left( \frac{\sqrt{2}\varepsilon}{\gamma c_1 \sqrt{\alpha}} + e^{- \Omega(\kappa^2)}\right), \quad i \in \set{1,\dots n-\abs{I(y)})}
\end{equation}
and 
\begin{equation}\label{incompmeanshiftC}
    \mathcal{L}((C^{\mathrm{T}} y'')_i, \sqrt{2}\varepsilon') 
    \leq C_{\ref{incompanticonc}} \left(\frac{\sqrt{2}\varepsilon'}{\gamma c_1 \sqrt{\alpha}} + e^{- \Omega(\kappa^2)}\right), \quad i \in \set{1,\dots p-\abs{I(x)}},
\end{equation}
where \(C_{\ref{incompanticonc}}\) depends only on \(m_4\), as remarked after Lemma \ref{incompanticonc}.

Using \(\left| [n]\setminus I(y) \right| \geq n / 2\) and \(\left| [p] \setminus I(x)\right| \geq p / 2 \), and the Tensorization Lemma (Lemma \ref{tensorization}) this yields
\begin{equation*}
    \mathcal{L}(B x'', \varepsilon \sqrt{n} )
    \leq \mathcal{L}(B x'', \sqrt{2}\varepsilon \sqrt{\abs{n \setminus I(y)}}) 
    \leq C_{\ref{incompanticonc}}^n \left(\frac{\sqrt{2}\varepsilon}{\gamma c_1 \sqrt{\alpha}} + e^{- \Omega(\kappa^2)}\right)^{\abs{n \setminus I(y)}}
\end{equation*}
and
\begin{equation*}
    \mathcal{L}(C^{\mathrm{T}} y'', \varepsilon' \sqrt{p} ) 
    \leq \mathcal{L}(C^{\mathrm{T}} y'', \sqrt{2}\varepsilon' \sqrt{\abs{p \setminus I(x)}}) 
    \leq C_{\ref{incompanticonc}}^p \left( \frac{\sqrt{2}\varepsilon'}{\gamma c_1 \sqrt{\alpha}} + e^{- \Omega(\kappa^2)} \right)^{\abs{p \setminus I(x)}}.
\end{equation*}
Multiplying these two upper bounds gives the desired result.
\end{proof}

\subsection{Epsilon nets for regularized LCD level sets and the union bound argument}

We first recall a lemma (\cite[Lemma 11.1]{nguyentaovugaps}) that bounds the size of regularized LCD level sets.

\begin{lemma}\label{lem:rlcd_level_net}
    For \(m^{-c} \leq \alpha \leq c'/4\) and any \(D \geq 1\), the level set
    \[
    S_{D,m} = \set{ x \in \incomp_m(c_0,c_1) \colon \widehat{\LCD}_{\kappa,\gamma}(x,\alpha) \leq D}
    \]
    has a \(\beta\)-net \(\mathcal{N}\) with \(\beta = \frac{\kappa}{\sqrt{\alpha} D}\) such that
    \begin{equation}
      \abs{\mathcal{N}} \leq  \frac{(CD)^m}{\sqrt{\alpha m}^{c'm/2}} D^{2/\alpha} \label{eq:net upper bound}  
    \end{equation}
    for \(C\) a constant depending on \(c_0\) and \(c_1\). 
\end{lemma}

We shall use the above result, combined with our anti-concentration results for singular vectors to prove that with high probability there are no singular vectors in \(S_D\). However, singular vectors with smaller LCD exhibit weaker anti-concentration. Fortunately, such vectors admit a net of smaller cardinality.

\begin{definition}\label{strata}
We define the set
\[
S_{D,k,m} = \set{ x \in \incomp_m(c_0,c_1) \colon 2^{-(k+1)} D \leq \widehat{\LCD}_{\kappa,\gamma}(x,\alpha) \leq 2^{-k} D },
\]
where \( 0 \leq k \leq \log_2 D\) (since for every  \(x\in \incomp_m(c_0,c_1)\), \(\widehat{\LCD}_{\kappa, \gamma}(x) \geq 1\)).
\end{definition}

We often use $S_D$ and $S_{D,k}$ if the dimension is clear from context. 
By replacing \(D\) with \(D/2^k\) in \eqref{eq:net upper bound}, we see that \(S_{D,k,m}\) has a \(\beta_k := 2^{k} \kappa / \sqrt{\alpha}D\)-net \(\mathcal{N}_k\) of cardinality
\begin{equation}
    \abs{\mathcal{N}_k} \leq \frac{(C D)^m}{(2^k)^m (\sqrt{\alpha m})^{c'm/2}} (2^{-k} D)^{2 / \alpha}. \label{eq:SDkm net}
\end{equation}

\begin{remark}
Note that the net given by Lemma \ref{lem:rlcd_level_net} may not consist of points that are in \(S_{D,k,m}\). However, we may remedy this by covering each \(\beta_k\) ball by \(6^m\) balls of radius \(\beta_k/2\) (see, for example,  Corollary 4.2.13 of \cite{vershynin-hdp}). Then, for each of these radius \(\beta_k/2\) balls that intersect \(S_{D,k,m}\) we may choose a point in the intersection and make this the center of a ball of radius \(\beta_k\), therefore giving a \(\beta_k\) net of \(S_{D,k,m}\). We absorb this factor of \(6\) into the constant \(C\) appearing in \eqref{eq:net upper bound} and \eqref{eq:SDkm net}.
\end{remark}

Instead of attempting a \(\beta\)-net argument for all of \(S_D,m\) for \(m=n\) or \(m=p\) at once, we instead break the problem down into examining each level set \(S_{D,k,m}\) separately, before union bounding over the possible values of \(k\).

Before proving our next proposition, we provide an equivalent formulation of the singular vector equations
\begin{align}
\Sigma^{1/2} Mv &= \sigma u \label{eq: sv equation 1}, \\
M^{\mathrm{T}} \Sigma^{1/2} u &= \sigma v\label{eq: sv equation 2}, 
\end{align}
where \(u\) and \( v \) are unit vectors. As in the proof of Lemma \ref{compcomp}, let us call
\[
w = \frac{\Sigma^{1/2}u}{\norm{\Sigma^{1/2}u}}
\]
so that \(w\) is  a unit vector. Note that
\[
\norm{\Sigma^{-1/2}w} = \frac{\norm{u}}{\norm{\Sigma^{1/2}u}} = \frac{1}{\norm{\Sigma^{1/2}u}}
\]
and 
\[
u = \Sigma^{-1/2}w \norm{\Sigma^{1/2}u} = \frac{\Sigma^{-1/2}w}{\norm{\Sigma^{-1/2}w}}.
\]
Then \eqref{eq: sv equation 1} reads as
\begin{equation*}
\Sigma^{1/2}Mv = \sigma \frac{\Sigma^{-1/2}w}{\norm{\Sigma^{-1/2}w}}.
\end{equation*}
By dividing both sides of \eqref{eq: sv equation 2} by \(\norm{\Sigma^{1/2}u}\), we obtain
\begin{equation*}
M^{\mathrm{T}} w = \frac{1}{\norm{\Sigma^{1/2}u}} \sigma v = \sigma \norm{\Sigma^{-1/2}w} v.
\end{equation*}

This motivates the formulation of the following proposition, which shows that for a fixed value of \(\sigma_0\), with high probability there are no singular vector pairs \((u,v)\) with singular value \(\sigma_0\) such that \(v\) or \(w=\Sigma^{1/2}u/\norm{\Sigma^{1/2}u}\) have small regularized \(\LCD\).

\begin{proposition}\label{approx sv}
Let \( M \) and \( \Sigma \) satisfy Assumptions \ref{matrix_model} and \ref{opnorm-assumptions}. There exists a positive constant \(c\) depending on \(c_0\) and \(c_1\) chosen in Lemma \ref{compcomp}, and therefore on \(c_{\ref{matrix_model}},K,m_4\), and \(L\), such that the following holds. Let \(\kappa = n^{2c}\) and \(\gamma = 1/2\). Suppose that \( n^{-c} \leq \alpha \leq c'/4\) and \( D \leq n^{c/\alpha}\) (where \(c'\) is defined in \eqref{eq:c'}). Let \(a\) be any fixed vector in \(\RR^{n}\), \(b\) be any fixed vector in \(\RR^p\) and \(0 \leq \sigma_0 \leq K L \sqrt{n} \). Then for \(\beta = \frac{\kappa}{\sqrt{\alpha}D}\), the union of events
\begin{align*}
\begin{split}
    \mathcal{E}_v =  \{ \exists (w,v) \in \incomp_n (&c_0,c_1)   \times S_{D,p} 
    \colon \norm{\Sigma^{1/2}M v - \sigma_0 \frac{\Sigma^{-1/2}w}{\norm{\Sigma^{-1/2}w}} - a}  \leq \beta \sqrt{n}, \\ 
    &\norm{M^{\mathrm{T}} w - \left(\sigma_0 \norm{\Sigma^{-1/2} w} \right) v - b} \leq  \beta \sqrt{n}, \\
    & \widehat{\LCD}_{\kappa, \gamma}(v,\alpha) \leq \widehat{\LCD}_{\kappa, \gamma}(w, \alpha) \wedge \mathcal{E}_K \}
\end{split}
\end{align*}
and 
\begin{align*}
\begin{split}
    \mathcal{E}_w = \{ \exists (w,v) \in  S_{D,n} \times &  \incomp_p(c_0,c_1)
    \colon \norm{\Sigma^{1/2}M v - \sigma_0 \frac{\Sigma^{-1/2}w}{\norm{\Sigma^{-1/2}w}} - a}  \leq  \beta \sqrt{n}, \\ 
    &\norm{M^{\mathrm{T}} w - \left(\sigma_0 \norm{\Sigma^{-1/2} w} \right) v - b} \leq \beta \sqrt{n},\\
    &\widehat{\LCD}_{\kappa, \gamma}(v,\alpha) \geq \widehat{\LCD}_{\kappa, \gamma}(w, \alpha) \wedge \mathcal{E}_K \}
\end{split}
\end{align*}
satisfies the probability bound
\begin{equation} \label{eq:vw-smallLCD}
    \mathbb{P}(\mathcal{E}_v \vee \mathcal{E}_w) \leq 2  n^{-c' c_{\ref{matrix_model}} n/32}. 
\end{equation} 
\end{proposition}

\begin{proof}
    We first bound the probability the event \(\mathcal{E}_{v,k}\), for a fixed \(k\), which is defined to be the event that there exists \((w,v) \in \incomp_n (c_0,c_1)   \times S_{D,k,p} \) that satisfy the bounds in \(\mathcal{E}_v\).

    \begin{comment}Assume that \(\mathcal{E}_v\) occurs, and let \( (w,v) \) be as in the event \( \mathcal{E}_{v} \). Assume without loss of generality that \(v \in S_{D,k, p}\) for some \(k\). Notice this and the definition of \( \mathcal{E}_{v} \) imply that \(\widehat{\LCD}_{\kappa, \gamma}(w,\alpha) \geq \frac{D}{2^{k+1}}\). We refer to this event where such a \(v\) and \(w\) exist as \(\mathcal{E}_{v,k}\).\end{comment}

Let \(\beta_k := \frac{2^k \kappa}{\sqrt{\alpha}D}\). By Lemma \ref{lem:rlcd_level_net}, there exists a \(\beta_k\)-net for \(S_{D,k,p}\), \(\mathcal{M}_k\) of cardinality
\begin{equation}\label{eq:strata upper bound}
\abs{\mathcal{M}_k} \leq \frac{(C D)^p}{(2^k)^p (\sqrt{\alpha p})^{c'p/2}} (2^{-k} D)^{2 / \alpha}
\end{equation}
and such that \(\mathcal{M}_k \subset S_{D,k,p}\). 
Moreover, using the trivial volumetric estimate for an \(\varepsilon\)-net of \(S^{n-1}\) (see, e.g. \cite[Lemma 5.2]{vershynin2012nonasymptotic}), the set \(\set{w \in S^{n-1} \colon \widehat{\LCD}_{\kappa, \gamma}(w) \geq \frac{D}{2^{k+1}}}\) has a \(\beta_k\) net \(\mathcal{N}_k\) of cardinality
\begin{equation}\label{eq:sphere net}
\abs{\mathcal{N}_k} \leq \left( \frac{6 \sqrt{\alpha}D}{2^{k} \kappa} \right)^n,
\end{equation}
such that all points \(w_0 \in \mathcal{N}_k\) are in the set \(\set{w \in S^{n-1} \colon \widehat{\LCD}_{\kappa, \gamma}(w) \geq \frac{D}{2^{k+1}}}\). 

Suppose that the event \(\mathcal{E}_{v,k}\) occurs and let \( (w,v) \in \incomp_n (c_0,c_1)   \times S_{D,k,p} \) that satisfy the bounds in \(\mathcal{E}_v\).  Then, there exist \(v_0 \in \mathcal{M}_k\) and \(w_0 \in \mathcal{N}_k\) such that
\begin{equation*}
\norm{v - v_0} \leq \beta_k, \qquad
\norm{w - w_0} \leq \beta_k.
\end{equation*}
It then follows that
\begin{align}
&\norm{\Sigma^{1/2}M v_0 - \sigma_0 \frac{\Sigma^{-1/2}w_0}{\norm{\Sigma^{-1/2}w_0}} -a} \nonumber\\
&\hspace{0.3\linewidth} \leq \norm{\Sigma^{1/2}M(v-v_0)} + \sigma_0 \norm{ \frac{\Sigma^{-1/2}w}{\norm{\Sigma^{-1/2}w}} - \frac{\Sigma^{-1/2}w_0}{\norm{\Sigma^{-1/2}w_0}}} \nonumber \\
&\hspace{0.33\linewidth} +\norm{\Sigma^{1/2}M v - \sigma_0 \frac{\Sigma^{-1/2}w}{\norm{\Sigma^{-1/2}w}} -a} \nonumber \\
&\hspace{0.3\linewidth} \leq \beta_k L \norm{M} + 2KL^{3} \beta_k \sqrt{n} + \beta \sqrt{n} \nonumber\\
&\hspace{0.3\linewidth} = O(\beta_k \sqrt{n}), \label{eq:LCD small ball v_0}
\end{align}
where the \(2KL^{3} \beta_{k}\sqrt{n}\) factor came from the following manipulation of the middle term 
\begin{align*}
&\sigma_0 \norm{ \frac{\Sigma^{-1/2}w}{\norm{\Sigma^{-1/2}w}} - \frac{\Sigma^{-1/2}w_0}{\norm{\Sigma^{-1/2}w_0}}} \\
& \leq \frac{\sigma_0}{\norm{\Sigma^{-1/2}w}\norm{\Sigma^{-1/2}w_0}} \norm{ \Sigma^{-1/2}w ||\Sigma^{-1/2}w_0||  - \Sigma^{-1/2}w_0 ||\Sigma^{-1/2}w|| } \\
&\leq \frac{\sigma_0}{\norm{\Sigma^{-1/2}w}} \norm{\Sigma^{-1/2}w - \Sigma^{-1/2}w_0} + \frac{\sigma_0}{\norm{\Sigma^{-1/2}w}} \abs{ ||\Sigma^{-1/2}w_0|| - ||\Sigma^{-1/2}w||}  \\
&\leq \frac{2 \sigma_0 \norm{\Sigma^{-1/2}}}{\norm{\Sigma^{-1/2}w}} \norm{w - w_0} \\
&\leq (2 KL\sqrt{n}) (L) (\beta_k) / L^{-1} \\
&= 2KL^{3} \beta_{k}\sqrt{n},
\end{align*}
\begin{comment}where the reverse triangle inequality has been used in the third-from-last line. \end{comment}
Similarly, we have that
\begin{align}
    &\norm{M^{\mathrm{T}} w_0 - \sigma_0 (\| \Sigma^{-1/2}w_0 \|) v_0 -b}
    \leq 4KL^2 \beta_k \sqrt{n} = O(\beta_k \sqrt{n}). \label{eq:LCD small ball w_0}
\end{align}

\begin{details}
    Full details: 
    \begin{align*}
        &\norm{M^{\mathrm{T}} w_0 - \sigma_0 (\| \Sigma^{-1/2}w_0 \|) v_0 -b}
        \\
        &\leq \left\|M^{\mathrm{T}}(w_0 - w)\right\| - \left\| \sigma_0 \left\|\Sigma^{-1 / 2}(w_0 - w)\right\|v_0\right\| + \left\| \sigma_0 \left\|\Sigma^{-1 / 2}w\right\|(v_0 - v)\right\| \\
        &\qquad + \left\| M^{\mathrm{T}}w - \sigma_0\left\|\Sigma^{-1 / 2}w\right\|v - b\right\|\\
        &\leq K\beta_k\sqrt{n} + KL^2 \beta_k\sqrt{n} + KL^2 \beta_k\sqrt{n} + \beta\sqrt{n}\\
        &\leq 4KL^2\beta_k\sqrt{n}
        = O(\beta_k \sqrt{n}).
    \end{align*}
\end{details}

Therefore, by a simple union bound, 
\begin{align*}
\mathbb{P}(\mathcal{E}_{v,k}) &\leq \abs{\mathcal{M}_k \times \mathcal{N}_{k}}  \\
&\qquad \times \max_{(v_0, w_0) \in \mathcal{M}_k \times \mathcal{N}_{k}} \mathbb{P}\Bigg(\norm{\Sigma^{1/2}M v_0 - \sigma_0 \frac{\Sigma^{-1/2}w_0}{\norm{\Sigma^{-1/2}w_0}} -a} = O(\beta_k \sqrt{n})), \\
&\qquad \qquad \qquad \qquad \qquad\norm{M^{\mathrm{T}} w_0 - \sigma_0 (\| \Sigma^{-1/2}w_0 \|) v_0 -b} = O \left( \beta_k \sqrt{n} \right) \Bigg)
\end{align*}
Now, we observe that for \(n\) sufficiently large (as \(\kappa=n^{2c}\)), 
\begin{align*}
\beta_k 
= \frac{\kappa 2^k}{\sqrt{\alpha}D} 
\geq \frac{\sqrt{\alpha}}{c_0 \widehat{\LCD}_{\kappa, \gamma}(v_0,\alpha)}
\end{align*}
and 
\begin{align*}
    \beta_k 
= \frac{\kappa 2^k}{\sqrt{\alpha}D} 
\geq \frac{\sqrt{\alpha}}{c_0 \widehat{\LCD}_{\kappa, \gamma}(w_0,\alpha)}.
\end{align*}
Therefore, for \((v_0, w_0) \in \mathcal{M}_k \times \mathcal{N}_{k}\) we may apply Lemma \ref{anti-concentration argument} with \(x = v_0\), \(y=w_0\), \(\varepsilon, \varepsilon' \geq \beta_k\) being the constants appearing in \eqref{eq:LCD small ball v_0} and \eqref{eq:LCD small ball w_0}, and replacing \(a\) and \( b \) in said lemma with \( a + \sigma_0 \frac{\Sigma^{-1/2}w_0}{\norm{\Sigma^{-1/2}w_0}}\) and \( b + \sigma_0 (\| \Sigma^{-1/2}w_0 \|) v_0\) here. Hence, we have that \( \PP(\mathcal{E}_{v,k})\) is bounded by
\begin{equation*}
\abs{\mathcal{M}_k \times \mathcal{N}_{k}} C_{\ref{anti-concentration argument}}^{n+p} \left[\ O \left( \frac{\beta_k}{\gamma c_1 \sqrt{\alpha}} + e^{- \Omega(\kappa^2)} \right) \right]^{n - \ceil{\alpha n}}  \left[O\left( \frac{\beta_k}{\gamma c_1 \sqrt{\alpha}} + e^{- \Omega(\kappa^2)} \right) \right]^{p - \ceil{\alpha p}}.
\end{equation*}
Using \eqref{eq:strata upper bound} and \eqref{eq:sphere net}, using that \(n^{-c} \leq \alpha \leq c'/4\) and \(D \leq n^{c/\alpha}\), we see that this is bounded by
\begin{align}
\nonumber \left( \frac{C \sqrt{\alpha}D}{2^{k} \kappa} \right)^n &\frac{(C D)^p}{(2^k)^p (\sqrt{\alpha p})^{c'p/2}} (2^{-k} D)^{2 / \alpha} C_{\ref{anti-concentration argument}}^{n+p} \left[\ O \left( \frac{\beta_k}{\gamma c_1 \sqrt{\alpha}} + e^{- \Omega(\kappa^2)} \right) \right]^{n - \ceil{\alpha n}}\\
\nonumber &\qquad \times   \left[O\left( \frac{\beta_k}{\gamma c_1 \sqrt{\alpha}} + e^{- \Omega(\kappa^2)} \right) \right]^{p - \ceil{\alpha p}} \\
\label{eq:net prob bound} &\leq \frac{(\tilde{C}^{n+p}) (D^{\left\lceil \alpha n\right\rceil + \left\lceil \alpha p\right\rceil +(2/\alpha)}) (\kappa^p)}{(2^{2k/\alpha})(\alpha^{ c'p/4 +p +n/2 -\lceil \alpha n \rceil - \lceil \alpha p \rceil})(p^{c'p/4})}  \\
\nonumber &\leq \frac{(\tilde{C}^{2n}) (D^{3\alpha n+(2/\alpha)}) (\kappa^n) \left( n^{c(c'n/4+3n/2)} \right)}{(c_{\ref{matrix_model}}n)^{c' c_{\ref{matrix_model}}n/4}} \\
\nonumber &= O \left( \frac{ n^{c(3 n+(2/\alpha^2))}  n^{c(c'n/4+7n/2)} }{n^{c' c_{\ref{matrix_model}}n/8}} \right) \\
\nonumber &= O \left( \frac{ n^{c((7+c'/4)n +2n^{2c})}}{n^{c' c_{\ref{matrix_model}}n/8}} \right) \\
\nonumber &= O \left( \frac{ n^{10cn}}{n^{c' c_{\ref{matrix_model}}n/8}} \right) \\
\nonumber &= O\left( \frac{1}{n^{c' c_{\ref{matrix_model}}n/16}} \right),
\end{align}
provided that \(c \leq c' c_{\ref{matrix_model}} / 160 \), where the implied constants depend on \(c\) and \(c'\).

\begin{details}
    The upper bound on \(\alpha\) is just the largest it can be as per the definition of \(\widehat{\LCD}\).
\end{details}

We union bound over all \( \log_2D \) values for \(k\) corresponding to each event \(\mathcal{E}_{v,k}\), to get:
\begin{align*}
\PP(\mathcal{E}_v) &\leq \sum_k \PP(\mathcal{E}_{v,k}) \\
&= O\left( \frac{\log_2 D}{n^{c'c_{\ref{matrix_model}}n/16}} \right) \\
&= O\left( \frac{\log(n^{c/\alpha})}{n^{c'c_{\ref{matrix_model}}n/16}} \right) \\
&=  O\left(\frac{1}{n^{c'c_{\ref{matrix_model}}n/32}} \right),
\end{align*}
where again the implied constants depend on \(c\) and \(c'\), and we can take
\[
    c = c' c_{\ref{matrix_model}} / 160 \leq c' c_{\ref{matrix_model}} / 32.
\]

Bounding \( \mathbb{P}(\mathcal{E}_w)\) is proven by an entirely analogous argument, interchanging the roles of \(v\) and \(w\) (which only makes the probability bounds stronger, due to the fact that the roles of \(n\) and \(p\) will be interchanged in expressions such as the one on line \ref{eq:net prob bound}, leading to a larger denominator). 
By the same argument we have that \( \mathbb{P} (\mathcal{E}_{w}) = O(n^{-c'c_{\ref{matrix_model}}n/32})\), so that
\begin{equation*}
    \mathbb{P} \left( \mathcal{E}_{v} \vee \mathcal{E}_{w}\right) = O \left(\frac{1}{n^{c'c_{\ref{matrix_model}}n/32}}\right)
\end{equation*}
\end{proof} 
We use the above proposition to prove Lemma \ref{noincomp-smallLCD}, in a similar fashion as in Proposition \ref{nocompressible}. Before doing so, we introduce the following terminology in order to more easily state the following two results.
\begin{definition}\label{def:singular_vector_pair}
 We say that \((u,v) \in S^{n-1} \times S^{p-1}\) is a \emph{singular vector pair} if \(\Sigma^{1/2} Mv =\sigma_i(\Sigma^{1/2}M)u\) and \(M^{\mathrm{T}} \Sigma^{1/2} u = \sigma_i(\Sigma^{1/2}M) v\) for some \(i\). In this context, we define the vector \(w\) associated to the singular vector pair \((u,v)\) to be \(w= \Sigma^{1/2}u/\norm{\Sigma^{1/2}u}\).
\end{definition}

\begin{lemma}\label{noincomp-smallLCD}
Under the same assumptions as Proposition \ref{approx sv}, there exists a positive constant \(c_{\ref{noincomp-smallLCD}}\) depending on \(c_{\ref{matrix_model}},K,m_4\) and \(L\) such that
\begin{align*}
\PP &\Big( \left\{ \exists \text{ singular vector pair } (u,v) \colon (w,v) \in  \incomp_n(c_0,c_1) \times S_{D,p} \right.\\
    &\hspace{0.25\linewidth}\left. \text{ or } (w,v) \in  S_{D,n} \times \incomp_p(c_0,c_1)\right\} \wedge \mathcal{E}_K \Big)
\leq   2 e^{-c_{\ref{noincomp-smallLCD}} n}.
\end{align*}
\end{lemma}

\begin{proof}
Assume there is such a singular vector pair \((u,v)\). As we are intersecting with the event \(\mathcal{E}_K\), we may assume that \(\sigma_i(\Sigma^{1/2} M) \in [0, K L \sqrt{n}]\). We again use a net argument to reduce the problem to fixed $\sigma$.  We can choose a $\beta \sqrt{n}/L$ net of \([0,KL \sqrt{n}]\) by considering non-zero integer multiples of \(  \beta \sqrt{n}/L\) (Giving a net of size \( \lfloor L^2 K/\beta \rfloor\)).  For any \(\sigma_i (\Sigma^{1/2} M)\), there exists a \(\sigma_0\) in the net such that \(\abs{\sigma_i - \sigma_0} \leq  \beta\sqrt{n}/L\). Then
\begin{align*}
\norm{\Sigma^{1/2} Mv - \sigma_0 u} &\leq  \abs{\sigma_i - \sigma_0} \norm{u}  \leq \frac{\beta \sqrt{n}}{L},  \\
\norm{M^{\mathrm{T}} \Sigma^{1/2} u - \sigma_0 v} &\leq \abs{\sigma_i - \sigma_0} \norm{v} \leq \frac{\beta \sqrt{n}}{L}.
\end{align*}
This in turn implies that 
\begin{align*}
\norm{\Sigma^{1/2}Mv - \sigma_0 \frac{\Sigma^{-1/2}w}{\norm{\Sigma^{-1/2}w}}} &= \norm{\Sigma^{1/2} Mv - \sigma_0 u} \leq \frac{\beta \sqrt{n}}{L} \leq \beta \sqrt{n} \\
\norm{M^{\mathrm{T}} w - \left(\sigma_0 \norm{\Sigma^{-1/2} w} \right) v} &= \frac{1}{\norm{\Sigma^{1/2}u}} \norm{M^{\mathrm{T}} \Sigma^{1/2} u - \sigma_0 v} \leq L \frac{\beta \sqrt{n}}{L} = \beta \sqrt{n},
\end{align*}
where we used the relations between \(u\) and \(w\) and their norms derived immediately after equations \eqref{eq: sv equation 1} and \eqref{eq: sv equation 2}. By Proposition \ref{approx sv} we have that for fixed \(\sigma_0\) this intersection of events has probability \(2 n^{-c' c_{\ref{matrix_model}} n/8} \leq  2 e^{-\hat{c} n}\) for some constant \(\hat{c}\). So by union bounding over all \( K L^2/ \beta\) choices of \(\sigma_0\), we get that
\begin{align*}
\PP & \left( \set{ \exists \text{ a singular vector pair } (u,v) \colon (w,v) \in \incomp_n(c_0,c_1) \times S_{D,p}} \wedge \mathcal{E}_K \right) \\
&\leq 2\exp(-\hat{c} n) K L^2 \beta^{-1} \\
&=   2\exp(-\hat{c} n) K L^2 D \sqrt{\alpha} / n^{2c} \\
&= O( n^{c n^c} \exp(-\hat{c} n)) \\
&= O(\exp(-\hat{c} n/2)).
\end{align*}
We then note that the event where \(v \in \incomp_p(c_0,c_1), \ w \in S_{D,n}\) satisfies the same probability bound by \eqref{eq:vw-smallLCD},
\begin{align*}
\PP & \left( \set{ \exists \text{ a singular vector pair } (u,v) \colon (w,v) \in   S_{D,n} \times  \incomp_p(c_0,c_1)} \wedge \mathcal{E}_K \right) \\
&= O(\exp (-\hat{c} n / 2)).
\end{align*}
\end{proof}

Combining Lemma \ref{noincomp-smallLCD} with Proposition \ref{nocompressible}, we obtain the following, which states that when we are working on the event \(\mathcal{E}_K\) we may assume, up to an exponentially small probability, that all singular vectors are incompressible with large LCD.

\begin{theorem}\label{bothlargeLCD}
Let \( M \) and \( \Sigma \) be as in Assumptions \ref{matrix_model} and \ref{opnorm-assumptions} and assume \(K \geq 1\). Then there exists \(c_{\ref{bothlargeLCD}},c>0\) depending on \(\kappa,\gamma,c_{\ref{matrix_model}},K,m_4\) and \(L\) such that:
\begin{align*}
\PP &\big( \{ \text{There exists a singular vector pair \((u,v)\): one of \(w\footnotemark\) or \(v\) is either compressible} \\ &\text{or is incompressible with}
\text{  \(\widehat{\LCD}_{\kappa,\gamma}(v,\alpha)\) or \(\widehat{\LCD}_{\kappa,\gamma}(w,\alpha) \leq n^{c/\alpha} \)} \} \wedge \mathcal{E}_K \big) \leq  2 e^{-c_{\ref{bothlargeLCD}} n} 
\end{align*} 
\footnotetext{Recall Definition \ref{def:singular_vector_pair}.}
for \(c\) as in Proposition \ref{approx sv} and \(\alpha\) in the range \(n^{-c} \leq \alpha \leq c'/4 = c_0 c_1^2/16\) (recall that \(c_0\) and \(c_1\) were fixed in Lemma \ref{nocompressible}).
\end{theorem}
\section{Proofs of Main Results}\label{mainproofs}
We use the facts we have established regarding the incompressibility and regularized LCD of unit singular vectors to prove our first main result. 

\begin{theorem}\label{main1}
    Let \(M\) and \(\Sigma\) be as in Assumptions \ref{matrix_model} and \ref{opnorm-assumptions}. Then there exist positive constants \(c\) and  \(C_{\ref{main1}}\) depending on \(c_{\ref{matrix_model}},K,m_4\) and \(m_4\) such that for \(n^{-c} \leq \alpha \leq c\) and \(\delta \geq n^{-c/\alpha}\),
    \begin{equation}\label{eq:main1.1}
\sup_{1 \leq i \leq p-1} \PP \left( \sigma_{i}^2(\Sigma^{1 / 2}M) - \sigma_{i}^2(\Sigma^{1 / 2}M') \leq  \sigma_{i}(\Sigma^{1 / 2}M')\, \delta n^{-1/2} \wedge \mathcal{E}_{K} \right) \leq C_{\ref{main1}} \frac{\delta}{\sqrt{\alpha}} ,
\end{equation}
Where \(M'\) is one of the \(n \times (p-1)\) minors of \(M\). In particular, by Cauchy's interlacing theorem (see \cite{MR1477662} Corollary III.1.5, recalling that the singular values are the eigenvalues of \(\sqrt{M^T \Sigma M}\)), we have that
\begin{equation}\label{eq:main1.2}
\sup_{1 \leq i \leq p-1} \PP \left( \sigma_{i}^2(\Sigma^{1 / 2}M) - \sigma_{i+1}^2(\Sigma^{1 / 2}M) \leq  \sigma_{i+1}(\Sigma^{1 / 2}M)\, \delta n^{-1/2}  \wedge \mathcal{E}_{K} \right) \leq C_{\ref{main1}} \frac{\delta}{\sqrt{\alpha}} .
\end{equation}
\end{theorem}
\begin{proof}
Fix \(1 \leq i \leq p-1\), and let
\[
   \mathcal{E}_i=\set{\sigma_{i}^2(\Sigma^{1/2} M) - \sigma^2_{i+1}(\Sigma^{1/2}M') \leq  \sigma_{i}(\Sigma^{1/2} M')\, \delta n^{-1/2}} \wedge \mathcal{E}_K.
\]
Consider a singular vector pair  \((u, v)\) of \(\Sigma^{1/2}M\) corresponding to the singular value \(\sigma_i(\Sigma^{1/2}M)\). As in previous sections, we let \(w = \Sigma^{1 / 2}u / \left\|\Sigma^{1 / 2}u\right\|\).
As we intersected with \(\mathcal{E}_K\) in our definition of \(\mathcal{E}_i\), all events in our proof can be assumed to be intersected with \(\mathcal{E}_K\) as needed (which we often omit to improve readability). If \(\mathcal{E}_i\) occurs, we have by Theorem \ref{bothlargeLCD} that with probability at least \(1 - 2e^{-c_{\ref{bothlargeLCD}} n}\), both \(v\) and \(w\) are incompressible with regularized LCD greater than or equal to \(n^{c/\alpha}\). Moreover, a \(n \times (p-1)\) minor \(M'\) of \(M\) will have a singular vector pair \((\hat{u}, \hat{v})\) (and \(\hat{w}\)) associated with \(\sigma_{i}(\Sigma^{1 / 2}M')\) for which the same holds. 

We decompose \(M\) as
\begin{equation}\label{eq:M'Xdecomp}
M = 
\begin{pmatrix}
M' \mid  X
\end{pmatrix}
\end{equation}
where $X$ is the final column on \(M\).
Consider the equation \(M^{\mathrm{T}} \Sigma M v = \sigma_i^2(\Sigma^{1 / 2}M) v\). We decompose \(v\) as 
\[
v = \begin{pmatrix}
v' \\ b
\end{pmatrix},
\]
where \(v' \in \RR^{p-1}\) and \(b\) is a scalar. It follows that
\begin{equation*}
M^{\mathrm{T}} \Sigma M v = \begin{pmatrix}
M'^{\mathrm{T}} \Sigma M' & M'^{\mathrm{T}} \Sigma X \\
X^{\mathrm{T}} \Sigma M' & X^{\mathrm{T}} \Sigma X
\end{pmatrix} 
\begin{pmatrix}
v ' \\
b
\end{pmatrix}
=
\sigma_i^2(\Sigma^{1 / 2}M) 
\begin{pmatrix}
v ' \\
b
\end{pmatrix}.
\end{equation*}
The first \((n-1)\) rows of this equation then reads:
\begin{equation*}
M'^{\mathrm{T}} \Sigma M' v' + M'^{\mathrm{T}} \Sigma X b = \sigma_i^2(\Sigma^{1 / 2}M) v'.
\end{equation*}
Rearranging the above gives
\begin{equation}
\label{eq:ev eqn} M'^{\mathrm{T}} \Sigma M' v' - \sigma_i^2(\Sigma^{1 / 2}M) v'  =  - M'^{\mathrm{T}} \Sigma X b. 
\end{equation}
Now let us multiply both sides of \eqref{eq:ev eqn} on the left by \(\hat{v}^{\mathrm{T}}\), the transpose of the right singular vector of \(\Sigma^{1 / 2}M'\) with singular value \(\sigma_i(\Sigma^{1 / 2}M')\). From this, we obtain
\begin{equation*}
(\sigma_i^2(\Sigma^{1 / 2}M') - \sigma_i^2(\Sigma^{1 / 2}M)) \hat{v}^{\mathrm{T}} v' = - b \sigma_i(\Sigma^{1 / 2}M') \norm{\Sigma^{1/2}\hat{u}} \hat{w}^{\mathrm{T}} X 
\end{equation*}
(recalling that \(\Sigma M' \hat{v} = \Sigma^{1/2} \sigma_i(\Sigma^{1/2} M') \hat{u} = \sigma_i(\Sigma^{1/2} M') \norm{\Sigma^{1/2}\hat{u}} \hat{w}\), where \(\hat{w} = \Sigma^{1/2}\hat{u}/\norm{\Sigma^{1/2}\hat{u}}\)).
Since \(\hat{v}\) and \(v'\) both have norm at most \(1\), applying the Cauchy-Schwarz inequality and the fact that \(\norm{\Sigma^{-1/2}} \leq L\) yields
\begin{equation*}
\abs{b \hat{w}^{\mathrm{T}} X} \sigma_i(\Sigma^{1 / 2}M') \leq L \abs{\sigma_i^2(\Sigma^{1 / 2}M) - \sigma_{i}^2(\Sigma^{1 / 2}M')}.
\end{equation*}
Therefore, if \(\mathcal{E}_i\), holds then
\begin{equation*}
\abs{b \hat{w}^{\mathrm{T}} X} \sigma_i(\Sigma^{1 / 2}M')  \leq L \sigma_{i}(\Sigma^{1 / 2}M') \abs{ \delta n^{-1/2}}.
\end{equation*}
The fact that \(\mathcal{E}_i\) implies \(\mathcal{E}_K\) mean that \(\Sigma^{1/2}M\), and therefore \(\Sigma^{1/2}M'\), have full rank, allowing us to divide by \(\sigma_i(\Sigma^{1/2}M')\). This yields
\begin{equation}
\label{eq:final coordinate} \abs{b \hat{w}^{\mathrm{T}} X} \leq L \abs{ \delta n^{-1/2}}.   
\end{equation}
Note that we wish to divide by \(b\), a coordinate of \(v\), a right singular vector for \(\Sigma^{1 / 2}M\). We wish to show that with high probability \(b\) is not too small, so that \ref{eq:final coordinate} provides a useful bound on the quantity \(\abs{\hat{w}^{\mathrm{T}} X}\). With this, we can apply the structure results proven in Section \ref{incompressible section} to \(\hat{w}\) and consequently use the small ball probability to bound \(\mathbb{P}(\mathcal{E}_i)\). 

By Proposition \ref{nocompressible}, \(v\) is incompressible with probability at least \(1 - 2 e^{-c_{\ref{nocompressible}} n}\). If \(v\) is incompressible then at least \(c_0 c_1^2 p /2 \geq c_0 c_1 c_{\ref{matrix_model}} n /2 \) coordinates of \(v\) have magnitude at least \(\frac{c_1}{\sqrt{2}} p^{-1/2} \geq \frac{c_1}{\sqrt{2}} n^{-1/2}\) (see \cite{rudelson2008littlewoodoffordprobleminvertibilityrandom}, Lemma 3.4). Let us call this lower bound \(B = \frac{c_1}{\sqrt{2}} n^{-1/2}\).

Let \(\mathcal{G}_{ij}\) be defined as
\[
\mathcal{G}_{ij} = \mathcal{E}_i \wedge \set{\abs{v^j} \geq B},
\]
where \(v^j\) is the \(j^{\text{th}}\) coordinate of \(v\).

Define the random variable
\[
p_B := \abs{ \set{ j \colon \abs{v^j} \geq B }}.
\]
Then \( p_B = l_1 + \dots + l_p \), where:
\[
l_j = \begin{cases}
1 \text{ if } \abs{v^j} \geq B \\
0 \text{ if } \abs{v^j} < B 
\end{cases}.
\]
It follows that for any \(N \geq 1\)
\begin{align*}
\PP(\mathcal{E}_i) &= \PP(\mathcal{E}_i \wedge \set{p_B \geq N}) + \PP(\mathcal{E}_i\, \wedge \set{p_B < N})   \\
&\leq \PP( \mathds{1}_{\mathcal{E}_i}\, p_B \geq N)  + \PP(p_B < N).
\end{align*}
Let us examine \(\PP( \mathds{1}_{\mathcal{E}_i}\, p_B \geq N)\). By Markov's inequality:
\begin{align*}
\PP( \mathds{1}_{\mathcal{E}_i}\, p_B \geq N) &\leq \frac{1}{N} \EE \left[ \mathds{1}_{\mathcal{E}_i}\, p_B \right] \\
&= \frac{1}{N} \sum_{j=1}^p \EE(l_j \, \mathds{1}_{\mathcal{E}_i}) \\
&\leq \frac{p}{N} \max_{1 \leq j \leq p} \EE(l_j \, \mathds{1}_{\mathcal{E}_i})\\
&= \frac{p}{N} \max_{1 \leq j \leq p} \PP(\mathcal{G}_{ij}).
\end{align*}
The \(\mathcal{G}_{ij}\) which has the largest probability will depend on the covariance matrix \(\Sigma\). However, as \(\Sigma\) is deterministic, we assume henceforth that \(\mathbb{P}(\mathcal{G}_{ip})\) is the greatest, in order to correspond to our decomposition \eqref{eq:M'Xdecomp}. For if this is not the case and \(\mathcal{G}_{ik}\) (\(k\neq p\)) achieves the maximum for the given \(\Sigma\), we would let \(X\) equal the \(k^{\text{th}}\) column of \(M\) and the rest of the argument would proceed identically. 

It therefore follows that
\begin{align}
\nonumber \PP(\mathcal{E}_i) &\leq \PP( \mathds{1}_{\mathcal{E}_i}\, p_B \geq N)  + \PP(p_B < N) \\
\nonumber &\leq \frac{p}{N} \max_{1 \leq j \leq p} \PP(\mathcal{G}_{ij}) + \PP(p_B < N) \\
\nonumber &= \frac{p}{N} \PP(\mathcal{G}_{ip}) + \PP(p_B < N) \\
&\leq \frac{p}{N} \PP \left(\abs{\hat{w}^{\mathrm{T}} X}  \leq L  \frac{ \delta n^{-1/2}}{B} \right) + \PP(p_B < N), \label{eq:Double counting}
\end{align}
where the final line uses the fact that
\(\mathcal{G}_{ip}\) implies
\[\abs{\hat{w}^{\mathrm{T}} X} \leq L \frac{ \delta n^{-1/2}}{B}.\] 
(This follows from \eqref{eq:final coordinate} and the definition of \(\mathcal{G}_{ij}\)).
Now choose \(N = c_0 c_1^2 c_{\ref{matrix_model}} n/4\) so that \(p/N \leq 4/ c_0 c_1^2 c_{\ref{matrix_model}}\), which is of constant order. Then, by this choice of \(N\) and the fact that \(B = \frac{c_1}{\sqrt{2}} n^{-1/2}\), we see that
\begin{equation*}
    \PP(p_B < N) \leq \PP(v \text{ is compressible}) \leq 2 e^{-c_{\ref{nocompressible}} n},
\end{equation*}
where the first inequality comes from the aforementioned fact that if \(v\) is incompressible , then \(p_B \geq N\) for our choice of \(B\) and \(N\) (see \eqref{eq:well spread} and \eqref{eq:c'}), and the second comes from Proposition \ref{nocompressible} (recalling that we are intersecting with \(\mathcal{E}_K\)). 
Using the above bound and plugging the values of \(N\) and \(B\) into \eqref{eq:Double counting} yields
\begin{align}\label{eq:smallballbound}
\PP(\mathcal{E}_i) &\leq \frac{2}{c_0 c_1^2 c_{\ref{matrix_model}}} \PP \left(\abs{\hat{w}^{\mathrm{T}} X} \leq L\frac{\sqrt{2} \delta}{c_1} \right) + 2 e^{-c_{\ref{nocompressible}} n}.
\end{align}
Let us examine the event \(\PP \left(\abs{\hat{w}^{\mathrm{T}} X} \leq L\sqrt{2} \delta / c_1 \right)\), which is bounded by
\[
   \mathcal{L}(\hat{w}^{\mathrm{T}}X, L\sqrt{2} \delta/ c_1).
\]
We know that with probability at least \(1 -  2 e^{- c_{\ref{bothlargeLCD}} n}\) that \(\hat{w}\) has \(\widehat{\LCD}_{\kappa, \gamma}(\hat{w},\alpha) \geq n^{c/\alpha}\). Therefore, we apply Lemma \ref{incompanticonc} on small ball probabilities in terms of regularized LCD. If \( \widehat{\LCD}_{\kappa, \gamma}(\hat{w}, \alpha) \geq n^{c / \alpha} \), then for \( \hat{C} \) large enough, we have
\begin{equation*}
\frac{\hat{C}L\sqrt{2} \delta}{c_{1}} \geq \frac{\hat{C}L\sqrt{2} n^{-c / \alpha}}{c_{1}} \geq \frac{\hat{C}L \sqrt{2}}{c_{1} \widehat{\LCD}_{\kappa, \gamma}(\hat{w}, \alpha)} \geq \frac{\sqrt{\alpha}}{c_{0} \widehat{\LCD}_{\kappa, \gamma}(\hat{w}, \alpha)}.
\end{equation*}
Therefore we can apply Lemma \ref{incompanticonc} with \(\kappa = n^{2c}\) and \(\gamma=1/2\), which yields
\begin{align}
    \PP \left(\abs{\hat{w}^{\mathrm{T}} X} \leq L\sqrt{2} \delta / c_1 \right) &\leq \mathcal{L}(\hat{w}^{\mathrm{T}} X, \hat{C}L \sqrt{2} \delta/ c_1) \label{eq:lastcolumnmeansum}\\
&\leq C_{\ref{incompanticonc}} \left( \frac{\hat{C}\sqrt{2}L \delta}{c_{1}^2 \gamma\sqrt{\alpha}} + e^{- \Omega(\kappa^2)} \right) + 2 e^{- c_{\ref{bothlargeLCD}} n} \label{eq:lastcolumnmeanshift}\\
\nonumber &= C_{\ref{incompanticonc}} \left(\frac{2\hat{C} \delta}{c_1^2\sqrt{\alpha}} \right) + C_{\ref{incompanticonc}}e^{- \Omega(n^{4c})} +  2e^{- c_{\ref{bothlargeLCD}} n}\\
\nonumber &= O \left(\frac{\delta}{\sqrt{\alpha}} \right).
\end{align}
Where we used that \(\delta \geq n^{-c/\alpha}\) to give us that
\[
    C_{\ref{incompanticonc}} e^{-\Omega(n^{4c})} + 2 e^{-c_{\ref{bothlargeLCD}} n} =  O \left(\frac{\delta}{\sqrt{\alpha}} \right).
\]
Similarly \( 2 e^{-c_{\ref{nocompressible}} n} = O(\delta/\sqrt{\alpha})\), so we conclude from \eqref{eq:smallballbound} that
\begin{equation*}
\PP(\mathcal{E}_i) \leq C_{\ref{main1}}\frac{\delta}{\sqrt{\alpha}}.
\end{equation*}

As \(1 \leq i \leq p-1 \) was arbitrary, we obtain the result \eqref{eq:main1.1}. As mentioned previously, an application of Cauchy interlacing immediately yields \eqref{eq:main1.2}. 
\end{proof}

\begin{corollary}\label{simple_singular_values}
    Let \( M \) and \(\Sigma\) be as in Assumptions \ref{matrix_model} and \ref{opnorm-assumptions}. Then there exists a positive constant \(c\) such that the singular values of $\Sigma^{1 / 2}M$ are simple with probability at least \(1 - 2\exp(-n^{c}) - \PP(\mathcal{E}_K^{c} )\).
\end{corollary}

\begin{proof}
    We apply Theorem \ref{main1} with $\alpha = n^{-c}$ and $\delta = n^{-c / \alpha} = n^{-c n^{c}}$.
    \begin{align*}
        \mathbb{P}&\left( \sigma_i(\Sigma^{1/2}M) = \sigma_{i+1}(\Sigma^{1/2}M) \text{ for some } i \wedge \mathcal{E}_K \right)\\
        &\leq (p-1) \sup_{1 \leq i \leq p-1} \mathbb{P}\left( \sigma_{i}(\Sigma^{1/2}M)^2- \sigma_{i+1}(\Sigma^{1/2}M)^2 = 0 \wedge \mathcal{E}_K \right)\\
        &\leq n \sup_{1 \leq i \leq p-1} \mathbb{P}\left( \sigma_{i}(\Sigma^{1/2}M)^2 - \sigma_{i+1}(\Sigma^{1/2}M)^2 \leq \sigma_i(\Sigma^{1/2}M')\delta n^{-1 / 2} \wedge \mathcal{E}_K\right)\\
        &= \ O\left(\frac{n \delta}{\sqrt{\alpha}} \right)\\
        &= O \left( n n^{-c n^{c}} n^{c / 2} \right)\\
        &= O (\exp(-n^{c})),
    \end{align*}
concluding the proof.
\end{proof}

In the proof of the next corollary, we will need the following bound on the least singular value in the case where we only assume that the atom variable has finite fourth moment and is not necessarily subgaussian. We slightly modify the original statement in order to account for \(\Sigma^{1/2}\), which we recall is assumed to have least singular value bounded below by \(L^{-1}\).
\begin{proposition}
    [Follows from Theorem 5.1 of \cite{rudelson2008littlewoodoffordprobleminvertibilityrandom}] \label{prop:small sv control}

    Let \(M\) and \(\Sigma\) satisfy Assumptions \ref{matrix_model} and \ref{opnorm-assumptions}. For any \(\delta \geq 0\) we let \(\mathcal{F} = \set{ \sigma_{p}(\Sigma^{1/2} M) \geq \delta n^{-1/2}}\). Then there exist positive constants \(C_{\ref{prop:small sv control}},c_{\ref{prop:small sv control}} >0\) depending only on \(c_{\ref{matrix_model}}\), \(K\), \(L\) and \(m_4\) such that
    \begin{equation*}
        \mathbb{P}(\mathcal{F}^c \wedge \mathcal{E}_K) \leq C_{\ref{prop:small sv control}} \delta +e^{-c_{\ref{prop:small sv control}}n}.
    \end{equation*}
\end{proposition}
Note that if one takes \(\delta=o(1)\) in the above proposition, one obtains that \(\mathbb{P}(\mathcal{F}^c \wedge \mathcal{E}_K) =o(1)\). We also remark that as originally stated in \cite{rudelson2008littlewoodoffordprobleminvertibilityrandom}, the theorem is a bound on \(\mathbb{P}(\mathcal{F}^c)\) with \(\mathbb{P}(\mathcal{E}_K^c)\) appearing in the bound. However, note that \cite{rudelson2008littlewoodoffordprobleminvertibilityrandom} Lemma 3.5 is also valid if one intersects the events appearing on both sides with \(\mathcal{E}_K\), while Lemma 3.3 and Theorem 5.2 are already stated in terms of intersection with \(\mathcal{E}_K\). The combination of these results gives Proposition \ref{prop:small sv control} as stated. 

\begin{details}
Essentially, Rudelson and Vershynin's proof of 
\[
\mathbb{P}(\mathcal{F}^c) \leq C\delta + c^n + \mathbb{P}(\mathcal{E}_K^c)
\]
is just proving 
\[
\mathbb{P}(\mathcal{F}^c \wedge \mathcal{E}_K) \leq C\delta + c^n
\]
and using that this implies their result via \(\mathbb{P}(\mathcal{F}^c) = \mathbb{P}(\mathcal{F}^c \wedge \mathcal{E}_k) + \mathbb{P}(\mathcal{F}^c \wedge \mathcal{E}_K^c)\). 
\end{details}

For \(1 \leq i \leq p-1\), we let
\[
    \delta_i = \sigma_i(\Sigma^{1 / 2}M) - \sigma_{i+1}(\Sigma^{1 / 2}M) \quad \text{ and } \quad \delta_{\text{min}} = \min_{1 \leq i \leq p-1} \delta_i.
\]
We now state the following corollary.
\begin{corollary}\label{cor:sing_value_diff}
   Let \(M\) and \( \Sigma \) satisfy Assumptions \ref{matrix_model} and \ref{opnorm-assumptions}, and \(C_{\ref{main1}}\), \(c\) be as in Theorem \ref{main1}. Let \(\mathcal{F}\) be as in Proposition \ref{prop:small sv control} and \(\mathcal{F}_i = \{\sigma_{i+1}(\Sigma^{1/2}M) \geq \delta n^{-1 / 2}\}\). Then for all \(n^{-c} \leq \alpha \leq c\), \(\delta \geq n^{-c / \alpha}\)
    \begin{align*}
        \sup_{1 \leq i \leq p-1} \mathbb{P}\left( \delta_i \leq \frac{1}{3}\delta n^{-1 / 2} \wedge \mathcal{F}_i \wedge \mathcal{E}_K\right) 
        &\leq C_{\ref{main1}}\delta / \sqrt{\alpha},
    \end{align*}

    Choosing \(\alpha\) to be a small constant, we have that for all \(\delta \geq n^{-C_{\ref{cor:sing_value_diff}}}\),
    \begin{align*}
        \sup_{1 \leq i \leq p-1}\mathbb{P}\left( \delta_i \leq \frac{1}{3}\delta n^{-1 / 2} \wedge \mathcal{F} \wedge \mathcal{E}_K\right) \leq C_{\ref{cor:sing_value_diff}} \delta,
    \end{align*}
    for some constant \(C_{\ref{cor:sing_value_diff}} \) depending only on \(c_{\ref{matrix_model}}\), \( m_{4} \), \(K\) and \(L\). 
    In particular, we have that
    \begin{align*}
        \delta_{\min} \geq n^{-3 / 2 - o(1)}
    \end{align*}
    with probability $1 - o(1)$.
\end{corollary}
\begin{proof}[Proof of Corollary \ref{cor:sing_value_diff}]
    Denote \(\sigma_i(\Sigma^{1/2}M)\) by \( \sigma_i\), and \(\sigma_i(\Sigma ^{1 / 2}M')\) by \(\sigma_i'\). 
    On the event $\mathcal{F}_i$ and $\delta_i \leq \frac{1}{3}\delta n^{-1 / 2}$, 
    \begin{align*}
        \frac{\sigma_i'}{\sigma_i + \sigma_i'} &= \frac{\sigma_i'}{2\sigma_i' + (\sigma_i - \sigma_i')} \geq \frac{\sigma_i'}{2\sigma_i' + \frac{1}{3}\delta n^{-1 / 2}} = \frac{1}{2 + \frac{\delta n^{-1 / 2}}{3 \sigma_i'}} \geq \frac{1}{3}.
    \end{align*}
    The last inequality follows from the observation that on $\mathcal{F}_i$, $\frac{\delta n^{-1 / 2}}{3 \sigma_i'} \leq \frac{\delta n^{-1 / 2}}{3 \sigma_{i+1}} \leq 1/3$ due to interlacing.
    Therefore,
    \begin{align*}
    \mathbb{P}\left( \delta_i \leq \frac{1}{3}\delta n^{-1 / 2} \wedge \mathcal{F}_i \wedge \mathcal{E}_K\right) &\leq \mathbb{P}\left( \delta_i \leq \frac{\sigma_i'}{\sigma_i + \sigma_i'} \delta n^{-1 / 2} \wedge \mathcal{F}_i \wedge \mathcal{E}_K\right) \\
    &\leq \mathbb{P}\left( \sigma_i^2 - \sigma_i'^2 \leq \sigma_i'\delta n^{-1 / 2} \wedge \mathcal{E}_K \right) \\
    &\leq C_{\ref{main1}}\delta / \sqrt{\alpha}
    \end{align*}
    where the final inequality follows from applying Theorem \ref{main1}.
    Combining this with the fact that \(\mathcal{F}\) implies \(\mathcal{F}_i\) for all $1 \leq i \leq p-1$ gives the result for \[ \sup_{1 \leq i \leq p-1} \mathbb{P}\left( \delta_i \leq \frac{1}{3}\delta n^{-1 / 2} \wedge \mathcal{F} \wedge \mathcal{E}_K \right). \]

    For the proof of the final statement, we take 
\(\delta = 3n^{-1-\varepsilon_n}\) (so that \( \frac{1}{3} \delta n^{-1 / 2} = n^{-3 / 2 -\varepsilon_n} \)), with \(\varepsilon_n = o(1)\) a sequence decaying slowly enough so that \(n^{-\varepsilon_n} = o(1)\) (take, for example, \(\varepsilon_n = \log(n)^{-1/2}\)) , and we take \( \alpha \) to be a small enough constant so that \(\delta \geq n^{-c/\alpha}\). We then apply a union bound to see that
    \begin{align*}
       &\mathbb{P}\left(\delta_{\text{min}} \leq \frac{1}{3}\delta n^{-1 / 2} \wedge \mathcal{F} \wedge \mathcal{E}_K\right) \\
        &\leq (p-1) \sup_{1 \leq i \leq p-1} \mathbb{P}\left(\delta_i \leq \frac{1}{3}\delta n^{-1 / 2} \wedge \mathcal{F} \wedge \mathcal{E}_K\right)\\
        &\leq n \, C_{\ref{main1}}\frac{\delta}{\sqrt{\alpha}} \\
        &= C_{\ref{cor:sing_value_diff}} n^{-\varepsilon_n} \\
	&= o(1).
    \end{align*} 
Now we apply Proposition \ref{prop:small sv control} with the above choice of \(\delta\) and the result \(\mathbb{P}(\mathcal{E}_K^c) = o(1)\) to conclude that
\begin{align*}
&\mathbb{P} \left( \delta_{\text{min}} \leq n^{-3/2 - o(1)} \right) \\
&=\mathbb{P} \left( \delta_{\text{min}} \leq \frac{1}{3} \delta n^{-1 / 2} \right) \\
&= \mathbb{P}\left(\delta_{\text{min}} \leq \frac{1}{3} \delta n^{-1 / 2} \wedge \mathcal{F} \wedge \mathcal{E}_K \right) + \mathbb{P}\left( (\mathcal{F} \wedge \mathcal{E}_K)^c \right) \\
&=o(1) + \mathbb{P}(\mathcal{F}^c \wedge \mathcal{E}_K) + \mathbb{P}(\mathcal{E}_K^c) \\
&=o(1).
\end{align*}
\end{proof}

In the case of a subgaussian atom variable, we may remove the term \( \PP(\mathcal{E}_K^{c} )\) from Corollary \ref{simple_singular_values} by combining the following two results. 

\begin{proposition}[Theorem 4.4.5, \cite{vershynin-hdp}]\label{prop:opnorm_subgaussian}
     Let \(M\) be an \(n \times p\) matrix for \(n \geq p\) with iid entries and subgaussian atom variable with mean zero and subgaussian moment \(q\). Then there exist positive constants \(C_{\ref{prop:opnorm_subgaussian}}\) and \(c_{\ref{prop:opnorm_subgaussian}}\) depending only on \(q\) such that 
     \[
     \mathbb{P}(\norm{M} > C_{\ref{prop:opnorm_subgaussian}} \sqrt{n}) \leq 2 e^{-c_{\ref{prop:opnorm_subgaussian}}n}.
     \]
\end{proposition}

\begin{proposition}[Theorem 1.1, \cite{rudelson2009smallestsingularvaluerandom}]\label{prop:sigma_n_subgaussian}
 Let \(M\) be an \(n \times p\) matrix for \(n \geq p\) with iid entries and subgaussian atom variable with mean zero and subgaussian moment \(q\). Then there exist positive constants  \(C_{\ref{prop:sigma_n_subgaussian}}, c_{\ref{prop:sigma_n_subgaussian}}\) depending only on \(q\) such that for every \(\varepsilon \geq 0\), we have
\[ \PP \left( \sigma_p(M) \leq \varepsilon(\sqrt{n} - \sqrt{p - 1} ) \right) \leq (C_{\ref{prop:sigma_n_subgaussian}}\varepsilon)^{n - p + 1}  + 2 e^{-c_{\ref{prop:sigma_n_subgaussian}}n}.\]
\end{proposition}
Taking \( \varepsilon = 0 \) in Proposition \ref{prop:sigma_n_subgaussian} (or alternatively using \cite[Theorem 1.1]{MR4255145}) and combining with Proposition \ref{prop:opnorm_subgaussian} gives that in the case of a subgaussian atom variable, \(\mathbb{P}(\mathcal{E}_K^c) = O(e^{-\tilde{c}n})\) for some \( \tilde{c} > 0 \) (depending on the subgaussian moment \(q\)). This fact along with Corollary \ref{simple_singular_values} immediately yields Corollary \ref{subgauss simple}.

Furthermore, in the case where \( M \) has subgaussian atom variable and is genuinely rectangular one can use Proposition \ref{prop:sigma_n_subgaussian} to further refine the result of Corollary \ref{cor:sing_value_diff} to obtain an exponential bound on \(\mathbb{P}(\mathcal{F}^c \wedge \mathcal{E}_K)\).

\begin{corollary}\label{cor:rectangular_implies_exp_prob} 
   Let \( M \) and \( \Sigma \) satisfy Assumptions \ref{matrix_model} and \ref{opnorm-assumptions} with subgaussian atom variable \(\xi\). Let \(c\) be as in Theorem \ref{main1}. If the aspect ratio \(\lambda\) satisfies $1 - \lambda \geq \frac{1}{\log\log n}$, then for \(n^{-c} \leq \alpha \leq c\) and \( n^{-c / \alpha} \leq \delta \leq 1\), one has
    \begin{align*}
        \sup_{1 \leq i \leq p-1}\mathbb{P}\left( \delta_i \leq \delta n^{-1 / 2} \right) = O(\delta / \sqrt{\alpha}) + 2\exp(-c_{\ref{cor:rectangular_implies_exp_prob}}n).
    \end{align*}
\end{corollary}
\begin{proof}
   Defining \(\nu = 1 - \lambda\) and using Proposition \ref{prop:sigma_n_subgaussian}, we have that
    \begin{align*}
       \mathbb{P}(\sigma_p(\Sigma^{1 / 2}M) \leq \delta n^{- 1 /2})
        &\leq \mathbb{P}\left( \sigma_p(M) \leq L\delta n^{-1 / 2} \right) + \mathbb{P}(\mathcal{E}_K^{c})\\
        &\leq \left( \frac{CL\delta}{\sqrt{n}(\sqrt{n} - \sqrt{p - 1})} \right)^{n - p + 1} + \exp(-c_{\ref{prop:sigma_n_subgaussian}}n)\\
        &\leq \left( \frac{CL\delta}{n(1 - \sqrt{\lambda})} \right)^{(1 - \lambda)n+1} + \exp(-c_{\ref{prop:sigma_n_subgaussian}}n)\\
        &\leq \left( \frac{CL\delta}{n(\nu / 2)} \right)^{\nu n} + \exp(-c_{\ref{prop:sigma_n_subgaussian}}n) 
    \end{align*}
    The choice of \(\delta\) and \( \nu \) guarantees that \(\left( \frac{CL\delta}{n(\nu / 2)} \right)^{\nu n} = o(\exp(-n))\), and that Corollary \ref{cor:sing_value_diff} applies. 
\end{proof}

\begin{remark}\label{rem:gen_main1}
    We remark that Theorem \ref{main1} can be extended to rank-1 perturbations of \(M\).  More precisely, in Theorem \ref{main1}, one can replace \(M\) by \(M + \eta \hat{y} \hat{z}^{\mathrm{T}}\), where \(\hat{y}, \hat{z}\) are deterministic unit vectors in \(\mathbb{R}^n\) and \(\mathbb{R}^p\), respectively, and \(\eta \in \mathbb{R}\) deterministic with \(|\eta| \leq \exp(o(n))\). Note that \(\|\eta \hat{y} \hat{z}^{\mathrm{T}}\| = |\eta|\).  As a consequence, Corollaries  \ref{simple_singular_values} and \ref{cor:sing_value_diff} both hold upon replacing \(M\) with \(M + \eta \hat{y} \hat{z}^{\mathrm{T}}\). The proof only requires minor modifications.  To illustrate this, we provide a complete proof of Lemma \ref{lem:generalized_compressible}, which is the analogue of Lemma \ref{compcomp} in the case of this finite rank perturbation.  We then sketch the analogous changes for the remainder of the argument.
  
	We consider two cases, depending on the size of $\eta$.
	
	\noindent {\bf Case 1: }	
    First, suppose that \(|\eta| \leq 3 L K \sqrt{n}\). In this case, the norm of \(M + \eta \hat{y} \hat{z}^{\mathrm{T}}\) is of the same order as $\|M\|$ and the entire argument is identical up to some slight changes to the constants.  For example, in each place where one considers a net of possible singular values of \(M\), an equally fine net of possible singular values of \(M + \eta \hat{y} \hat{z}^{\mathrm{T}}\) is now 4 times as large (or 16 times as large in the case of the squared singular values). Such a difference in net size is immaterial. In this case, we do not need to make use of the rank-1 condition, so in fact, the argument applies to any perturbation with sufficiently small operator norm. For more details, see Remark \ref{rem:mean non zero}.\\
    
    \noindent {\bf Case 2: }  We now consider the case that \(3LK \sqrt{n} \leq \abs{\eta} \leq \exp(o(n))\), omitting \(\Sigma\) for ease of presentation.  This presents several difficulties.  In this setting, simply replacing \(M\) with \(M + \eta \hat{y} \hat{z}^{\mathrm{T}}\) would require significantly finer nets to account for the larger norm, yet the probability bounds remain the same, which quickly renders most of the arguments infeasible.  There is good reason to expect this breakdown as for large \(\eta\), the dominant left and right singular vectors are near \(w\) and \(z\), which may not have the same properties as the singular vectors of a centered random matrix.       
    
    We begin by establishing a lower bound on the gap between $\sigma_1$ and $\sigma_2$.  
    On the event \( \left\|M\right\| \leq K\sqrt{n}\), by Weyl's inequality (\cite{MR1477662} Theorem III.2.1),
    \begin{align*}
    	\sigma_1(M + \eta \hat{y} \hat{z}^{\mathrm{T}}) \geq |\eta| - \left\|M\right\| \geq 2KL\sqrt{n}, \\
    	\sigma_2(M + \eta \hat{y} \hat{z}^{\mathrm{T}}) \leq \sigma_2(\eta \hat{y} \hat{z}^{\mathrm{T}}) + \left\|M\right\| \leq KL\sqrt{n},
    \end{align*}
    (note that the second inequality holds independent of the value of $|\eta|$, as \(\eta \hat{y} \hat{z}^{\mathrm{T}}\) is rank-1)
    so \(|\sigma_1(M) - \sigma_2(M)| \geq KL\sqrt{n}\). 
    In several steps of the proof of Theorem \ref{thm:main}, for example in Proposition \ref{nocompressible}, Lemma \ref{noincomp-smallLCD}, Proposition \ref{approx sv} and Theorem \ref{bothlargeLCD}, we assume that our singular values are of order $KL \sqrt{n}$.  Due to the substantial gap between $\sigma_1$ and $\sigma_2$, for the remainder of the argument, we can focus our attention on the singular values which have size at most \(KL \sqrt{n}\), which resolves one aspect of this problem.  For example, it suffices to prove Theorem \ref{bothlargeLCD} restricted to the case where \((u,v)\) is a singular vector pair associated to some singular value \(\sigma_i\) where \(2 \leq i \leq p\).

    However, even when restricting our attention to smaller singular values, the norm of \(M + \eta \hat{y} \hat{z}^{\mathrm{T}}\) still appears in our arguments.  For example, in Section \ref{incompressible section}, the covering arguments for a singular vector pair \((u,v)\) do not make use of any knowledge of other singular pairs, so we apply the argument to the entire unit sphere for each singular vector pair, in which case, naively, we must invoke the norm of \(M + \eta \hat{y} \hat{z}^{\mathrm{T}}\) when deciding on the size of the net.  To address this issue, we now rely on the rank-1 structure of $\eta \hat{y} \hat{z}^{\mathrm{T}}$.    
    
    Let \(T \subset S^{p-1}\) (for example,  \(T\) may be the set compressible vectors, or a set of incompressible vectors with small LCD).  The key observation is that the image of the unit sphere under $\eta \hat{y} \hat{z}^{\mathrm{T}}$ is contained in \( \{ t \hat{y}  : \abs{t}\leq \eta \}\) and therefore admits an $\varepsilon$-net of small cardinality due to its dimension. For example, to naively control 
    \[
    \mathbb{P}\left(\sup_{x \in T} \left\|(M + \eta \hat{y} \hat{z}^{\mathrm{T}})x \right\| < \varepsilon\right)
    \]
    we consider an \(\varepsilon(2 \|M + \eta \hat{y} \hat{z}^{\mathrm{T}}\|)^{-1}\)-net of \(T\), that we denote by \(\mathcal{N}\), and use the union bound to conclude that
    \[
    \mathbb{P}(\sup_{x \in T} \left\|(M + \eta \hat{y} \hat{z}^{\mathrm{T}})x \right\| < \varepsilon) \leq |\mathcal{N}| \sup_{x_0 \in \mathcal{N}} \mathbb{P}( \left\|(M + \eta \hat{y} \hat{z}^{\mathrm{T}})x_0 \right\| < 2 \varepsilon).
    \]  
    This is problematic since when \(\eta\) is very large, \(\mathcal{N}\) would have cardinality approximately \(\eta^p\) times larger than an \(\varepsilon\)-net of the same set \(T\). Yet, our anti-concentration bound will remain the same size.  However, we can instead consider an \(\varepsilon(2 K \sqrt{n})^{-1}\)-net of $T$, which we call $\mathcal{N}'$, and a $\varepsilon/2$-net of $[- \eta, \eta] \subset \mathbb{R}$, which we denote by $\mathcal{N}''$.  Now, we observe that if $x \in T$ is such that 
    \[
    \left\|(M + \eta \hat{y} \hat{z}^{\mathrm{T}})x \right\| < \varepsilon
    \]
    then there exists a $y \in \mathcal{N}'$ and $t \in \mathcal{N}''$ such that
    \begin{align*} 
    \| (M + \eta \hat{y} \hat{z}^{\mathrm{T}}) x - M y - t y\| \leq \|M\| \|x - y\| + |\eta z^{\mathrm{T}} x - t| \|y\| \leq \varepsilon.
    \end{align*}
	Thus, 
	 \[
	\mathbb{P}(\sup_{x \in T} \left\|(M + \eta \hat{y} \hat{z}^{\mathrm{T}})x \right\| < \varepsilon) \leq |\mathcal{N}'| |\mathcal{N}''|\sup_{x \in \mathcal{N}', t \in \mathcal{N}''} \mathbb{P}( \left\|M x + t y \right\| < 2 \varepsilon).
	\]  
Crucially, \(|\mathcal{N}'|\) is much smaller than \( |\mathcal{N}| \) and as \(\mathcal{N}''\) is a net over a one-dimensional space, its cardinality is insignificant (so long as the size of \( \eta \) is sufficiently small), leading to only a small increase in the number of points we must net over in our arguments.

To concretely illustrate this idea, we state and prove Lemma \ref{compcomp} in the case where one replaces \( M \) with \( M + \eta \hat{y} \hat{z}^{T} \).

\begin{lemma}\label{lem:generalized_compressible}
    Let \( M \) satisfy Assumption \ref{matrix_model} and \( \Sigma \) satisfy Assumption \ref{opnorm-assumptions}. Let \( \eta = \eta(n) \leq \exp (o(n)) \) and \( \hat{z} \in \mathbb{R}^{p}, \hat{y}\in \mathbb{R}^{n}\) be unit vectors. Define \( \hat{M} = M + \eta \hat{y} \hat{z}^{T} \) and the event \( \mathcal{E}_{K}' \) to be 
    \begin{equation*}
        \left\{\left\|M\right\| \leq K \sqrt{n} \right\} \wedge \left\{ \sigma_{p} (\hat{M}) \neq 0 \right\}.
    \end{equation*}
    Then we have that
\begin{equation}\label{mod_compeq-p}
\sup_{ 0 \leq \sigma^{2} \leq K^{2}L^{2} n} \PP \set{ \inf_{v \in \comp_p(c_0,c_1)} \norm{\hat{M}^{\mathrm{T}} \Sigma \hat{M} v - \sigma^2 v} \leq c_{\ref{lem:generalized_compressible}} n \wedge \mathcal{E}_K' } \leq  2 e^{-c_{\ref{lem:generalized_compressible}} n} 
\end{equation}
and 
\begin{align}\label{mod_compeq-n}
    \sup_{ 0 \leq \sigma^{2} \leq K^{2}L^{2} n} \PP \set{ \inf_{\Sigma^{1 / 2}u/ \norm{\Sigma^{1 / 2}u} \in \comp_n(c_0,c_1) } \norm{\Sigma^{1 / 2}\hat{M} \hat{M}^{\mathrm{T}} \Sigma^{1 / 2} u - \sigma^2 u} \leq c_{\ref{lem:generalized_compressible}} n \wedge \mathcal{E}_K' } \\
\nonumber \leq 2 e^{-c_{\ref{lem:generalized_compressible}} n},
\end{align}
for some sufficiently small constant \( c_{\ref{lem:generalized_compressible}} > 0 \).
\end{lemma}
\begin{proof}
    Let the value of \(\sigma^2 \in [0, K^2L^2n]\) remain fixed throughout the proof. We will assume that \( \mathcal{E}_{K}'\) holds, so that all events below are understood to be intersected with \( \mathcal{E}_{K}' \). As in the proof of Lemma \ref{compcomp}, we will let \( \mathcal{M} \) and \( \mathcal{N} \) be \( 12c_1 \) nets of \( \comp_p(c_0,c_1) \) and \( \comp_n(c_0,c_1) \) respectively, and we can assume that \( \mathcal{M} \) and \( \mathcal{N} \) consist only of sparse unit vectors. Furthermore, we will let \( \mathcal{R}_{1}, \mathcal{R}_{2}\) be \( c_1 \)-nets of \( \text{Span}(\hat{z})\cap B(0, \eta^2 L^2 K \sqrt{n}) \) and \( \text{Span}(\hat{y}) \cap B(0, \eta) \), respectively. Notice that
\begin{align}\label{eq:Mhat expansion}
    \hat{M}^{T} \Sigma \hat{M}v 
    &= M^{T} \Sigma (M v +  \hat{y} (\eta \hat{z}^{T} v)) + (\eta \hat{y}^T  \Sigma M v + \eta^{2} \hat{y}^T \Sigma \hat{y} \hat{z}^T v)\hat{z}.
\end{align}
Let \(z'= (\eta \hat{y}^T  \Sigma M v + \eta^{2} \hat{y}^T \Sigma \hat{y} \hat{z}^T v)\hat{z}\) and \(y' = (\eta \hat{z}^T v)\hat{y}\). Note that by the Cauchy-Schwarz inequality and Assumptions \ref{matrix_model} and \ref{opnorm-assumptions}, \(\norm{y'} \leq \abs{\eta}\), while \( \norm{z'} \leq \eta^2 L^2 K \sqrt{n}\). This implies that if there exists some \( v \in \comp_p(c_0, c_1) \) for which the event in equation \eqref{mod_compeq-p} holds, then there exists some \( v_0 \in \mathcal{M}, z_0 \in \mathcal{R}_{1}, y_0 \in \mathcal{R}_{2}, \) such that \(\norm{v_0-v} \leq c_1\), \(\norm{y_0-y'} \leq c_1\) and \(\norm{z_0-z'} \leq c_1\), and 
\begin{align*}
    &\left\|M^{T} \Sigma (Mv_0 + y_0 ) + z_0 - \sigma^2 v_0\right\| \\
    &\leq \norm{\hat{M}^{T} \Sigma \hat{M} v - \sigma^{2} v} + \norm{M^T \Sigma \left( M (v_0-v) + (y_0-y') \right)} +\norm{z_0-z'} + \sigma^2 \norm{v_0-v} \\
    &\leq c_{\ref{lem:generalized_compressible}} n + 12 c_1 L^2K^2 n + c_1 K L^2 \sqrt{n} + c_1 + 12 \sigma^2 c_1 \\
    &\leq \left( c_{\ref{lem:generalized_compressible}} + 26c_1 L^2 K^2 \right) n
\end{align*}
\begin{comment}
    \left\|M^{T} \Sigma M(v - v_0)\right\| + \left\|M^{T} \Sigma (w' - w(\eta z^{T} v))\right\| \\
    &\quad+ \left\|z' - z(\eta y^{T} \Sigma Mv + \eta^{2} z^{T} v)\right\| + \left\| (\sigma^2 - \sigma^2) v_0 \right\| + \left\| \sigma^2 (v - v_0)\right\|\\
    &\leq c_{\ref{lem:generalized_compressible}}n + K^{2} L^{2} 12 c_1 n + KL^{2}c_1\sqrt{n} + c_1 + c_1 n + K^{2} L^{2} 12 c_1 n \leq (c_{\ref{lem:generalized_compressible}} + 26K^{2}L^{2}c_1)n,
\end{comment}
Therefore, we can conclude the proof by bounding
\begin{equation*}
    \mathbb{P}\left\{\left\|M^{T} \Sigma (Mv_0 + y_0 ) + z_0 + \sigma^2 v_0\right\| \leq (c_{\ref{lem:generalized_compressible}} + 26 K^{2} L^{2} c_1)n \right\}.
\end{equation*}
for fixed \(v_0,y_0\) and \(z_0\), and then taking a union bound over the nets \(\mathcal{M}, \mathcal{R}_1\) and \(\mathcal{R}_2\). 
As in Lemma \ref{compcomp}, we may assume without loss of generality, that \( v_0 \) is supported on the first \( \left\lfloor c_0 p\right\rfloor \) coordinates. With this in mind, we rewrite the equation using the decomposition \( M = \begin{pmatrix} X & Y \end{pmatrix} \) where \(X\) is an \(n \times \lfloor c_0p \rfloor\) matrix and \(Y\) is an \( n \times \lceil(1-c_0)p \rceil \) matrix, so that
\begin{equation*}
    \left\|M^{T} \Sigma (Mv_0 + y_0 ) + z_0 + \sigma v_0\right\|
    = \left\| \begin{pmatrix} X^{T} \\ Y^{T} \end{pmatrix} \Sigma \left(\begin{pmatrix} X & Y \end{pmatrix} v_0 + y_0 \right) + z_0 + \sigma v_0\right\|.
\end{equation*}
Now, recalling that \( v_0 \) is only supported on the first \( \left\lfloor c_0 p\right\rfloor \) coordinates, we have that
\begin{equation*}
    \left(\begin{pmatrix} X & Y \end{pmatrix}v_0\right)_k = \sum_{i = 1}^{\left\lfloor c_0 p\right\rfloor} X_{ki} v_{0i}.
\end{equation*}
Since we have that each \( X_{ki} \) is iid and independent of \( Y \), this implies that \( \left(\begin{pmatrix} X & Y \end{pmatrix} v_0\right)_{k} \) are iid, and independent of \( Y \). If we denote \( v_0'\in \mathbb{R}^{\left\lfloor c_0 p\right\rfloor}  \) such that \( v_{0j} = v_{0j}'  \) for all \( j = 1, \ldots, \left\lfloor c_0 p\right\rfloor \) (so that \( (Mv_0)_{k} = (Xv_0' )_{k} \)), then by Lemma \ref{sum anticonc} we have:
\begin{equation*}
    \mathcal{L}\left((X v_0')_{k}, \frac{1}{2} \right)\leq \tilde{c},
\end{equation*}
for some constant \( \tilde{c} \) depending on \( m_4 \).
Then, the Tensorization Lemma can be applied to the \( n \) rows of \( X \) to give:
\begin{equation*}
    \mathcal{L}(X v_0', \varepsilon \sqrt{n} ) \leq e^{-c_{\ref{lem:generalized_compressible}}' n}
\end{equation*}
for some constants \( \varepsilon, c_{\ref{lem:generalized_compressible}}'  \) depending only on \( \tilde{c} \). By taking \( \hat{c}_{\ref{lem:generalized_compressible}} = \min \left\{c_{\ref{lem:generalized_compressible}}' , \varepsilon\right\} \), and looking at the concentration around \( -y_0  \), this tells us that
\begin{equation*}
    \mathbb{P}\left( \left\|X v_0' + y_0 \right\| \leq \hat{c}_{\ref{lem:generalized_compressible}} \sqrt{n} \right) \leq e^{-\hat{c}_{\ref{lem:generalized_compressible}} n}.
\end{equation*}
We can now condition on \( X \) and the event that \( \left\|X v_0' + y_0  \right\|  \geq  \hat{c}_{\ref{lem:generalized_compressible}} \sqrt{n}\), so that \( \left\| \Sigma(X v_0' + y_0) \right\| \geq L^{-2} \hat{c}_{\ref{lem:generalized_compressible}} \sqrt{n} \). Now, 
\begin{equation*}
    \Sigma(X v_0' + y_0)  / \left\| \Sigma(X v_0' + y_0) \right\|
\end{equation*}
is a unit vector. Denoting \( P \) the projection onto the last \( p - \left\lfloor c_0 p\right\rfloor \) coordinates, and \( x = \Sigma (Xv_0' + y_0 ) \),
\begin{align*}
    \left\|M^{T}( \Sigma (Xv_0' + y_0 )) + z_0 - \sigma^{2} v_0\right\| 
    &\geq \left\|P\left(M^{T}x + z_0 - \sigma^{2} v_0\right)\right\| \\
    &= \left\|Y^{T}x + Pz_0\right\|\\
    &= \left\| x \right\|\left\|Y^{T}( x / \left\|x\right\| ) + Pz_0 / \left\|x\right\|\right\|.
\end{align*}
Due to our conditioning, \( x \) is fixed 
and the same argument (this time looking at concentration around \( -Pz_0 / \left\|x\right\|  \)) yields:
\begin{equation*}
    \mathbb{P}\left( \left\|Y^{T}(x / \left\|x\right\|) + Pz_0 / \left\|x\right\| \right\| \leq \hat{c}_{\ref{lem:generalized_compressible}} \sqrt{\left\lceil p - c_0 p\right\rceil} \right) \leq e^{\hat{c}_{\ref{lem:generalized_compressible}} \left\lceil p - c_0 p\right\rceil}.
\end{equation*}
Therefore, we have that 
\begin{align*}
    \mathbb{P}\left(\left\|M^{T} \Sigma (Mv_0 + y_0 ) + z_0 + \sigma v_0\right\| \leq (c_{\ref{lem:generalized_compressible}} + 26 K^{2} L^{2} c_1)n \right) 
    &\leq e^{-\hat{c}_{\ref{lem:generalized_compressible}}n } + e^{-\hat{c}_{\ref{lem:generalized_compressible}}\left\lceil p - c_0 p\right\rceil} \\
    &\leq 2 e^{-\hat{c}_{\ref{lem:generalized_compressible}}c_{\ref{matrix_model}}n / 2 },
\end{align*}
so long as \( c_{\ref{lem:generalized_compressible}} \) and \( c_1 \) are taken small enough such that
\[
    c_{\ref{lem:generalized_compressible}} + 26 K^{2} L^{2} c_1 \leq \hat{c}_{\ref{lem:generalized_compressible}}^{2} L^{-2} \sqrt{c_{\ref{matrix_model}} / 2},
\]
which is possible since \( \hat{c}_{\ref{lem:generalized_compressible}} \) does not depend on \( c_1 \) or \( c_{\ref{lem:generalized_compressible}} \).

Overall, incorporating the necessary union bounds, this net argument yields
\begin{align*}
&\PP \set{ \inf_{v \in \comp_p(c_0,c_1)} \norm{\hat{M}^{\mathrm{T}} \Sigma \hat{M} v - \sigma^2 v} \leq c_{\ref{lem:generalized_compressible}} n \wedge \mathcal{E}_K' } \\
&\leq |\mathcal{M}| |\mathcal{R}_{1}| |\mathcal{R}_{2}| \\
& \times \quad\sup_{\substack{ v_0 \in \mathcal{M} \\ z_0 \in \mathcal{R}_1, y_0 \in \mathcal{R}_2}} \mathbb{P} \left(\left\|M^{T} \Sigma (Mv_0 + y_0 ) + z_0 + \sigma v_0\right\| \leq (c_{\ref{lem:generalized_compressible}} + 26 K^{2} L^{2} c_1)n \right)\\
&\leq \left(\frac{9}{c_0 c_1}\right)^{c_0 n}  \left( \frac{\eta^2 L^2 K \sqrt{n}}{c_1} \right) \left( \frac{2 \eta}{c_1} \right) 2 e^{-\hat{c}_{\ref{lem:generalized_compressible}} c_{\ref{matrix_model}} n / 2}\\
&\leq 2e^{-c_{\ref{lem:generalized_compressible}}n},
\end{align*}
by taking \( c_0, c_{\ref{lem:generalized_compressible}} \) sufficiently small (potentially making \( c_{\ref{lem:generalized_compressible}} \) smaller than previously chosen) and using the fact that \(\eta \leq e^{o(n)}\) to control the sizes of \(\mathcal{R}_1\) and \(\mathcal{R}_2\).

The argument for equation \eqref{mod_compeq-n} is similar, again following the spirit of Lemma \ref{compcomp}, while additionally netting over possible \( y \) and \( z \) as appropriate.  As in the proof of Lemma \ref{compcomp}, this time we decompose \( M \) as \( \begin{pmatrix} X \\ Y \end{pmatrix} \) rather than \( \begin{pmatrix} X & Y \end{pmatrix} \).

\end{proof}

    We now sketch the use of this idea in the remaining sections of our argument.

    For the incompressible case with \(M + \eta \hat{y} \hat{z}^{\mathrm{T}}\), note that Lemma \ref{anti-concentration argument} remains effectively unchanged, as \(\eta \hat{y} \hat{z}^{\mathrm{T}} x\) and  \(\eta z \hat{y}^{\mathrm{T}}y\) again reside in one-dimensional spaces that we can efficiently net over. Hence, up to a slight change in constants,  Proposition \ref{approx sv} also remains unchanged. Finally, this implies that Lemma \ref{noincomp-smallLCD} can be modified so that the event in question is
    \begin{align*}
        &\big\{ \exists \text{ singular vector pair } (u,v) \colon (w,v) \in \incomp_n(c_0,c_1)  \times S_{D,p} \\
        &\hspace{80pt}\text{ or } (w,v) \in S_{D,\RR^n}  \times \incomp_p(c_0,c_1) \wedge \mathcal{E}_K \text{ for some } i = 2, \ldots, p \big\}.
    \end{align*}
    The restriction to \(i \neq 1\) guarantees that the singular value for which the above event occurs is less than \(LK\sqrt{n}\) (assuming \(\mathcal{E}_K' \)). This implies that 
    for all \(2 \leq i \leq p\), we may assume \(\widehat{\LCD}_{\kappa,\gamma}(v,\alpha), \widehat{\LCD}_{\kappa,\gamma}(w,\alpha) \geq n^{c/\alpha} \) whenever we are working on the event \(\mathcal{E}_K'\).
    with possibly different constants. Finally, for the main result, Theorem \ref{main1}, we have that the probability in \eqref{eq:main1.1} (without the supremum) can be bounded for \(i=2, \ldots, p-1\), and hence \eqref{eq:main1.2} also holds for the same values of \(i\). However, for  \(i = 1\), the gap is of size  \(LK\sqrt{n}\) by Weyl's inequality (assuming \(\mathcal{E}_K' \)), and so \eqref{eq:main1.2} holds for \(i = 1, \ldots, p-1\), as desired. As for Corollaries \ref{simple_singular_values} and \ref{cor:sing_value_diff}, their proofs follow by similarly splitting the problem into considering \(\sigma_1 - \sigma_2\) and \(\sigma_j - \sigma_{j+1}\) (\(j \geq 2\)) separately, using Weyl's inequality for the first quantity, and (the generalized form of) Theorem \ref{main1} for the second.
    \end{remark}

    \begin{remark} \label{rem:extension}
        Finally, we address the necessary adjustments to extend Assumption \ref{matrix_model} to $c_{\ref{matrix_model}} \leq \lambda \leq c_{\ref{matrix_model}}^{-1}$. 
        The gaps between the singular values of \( \Sigma^{1/2} M \) can be deduced from the gaps of the singular values of \( M^{\mathrm{T}} \Sigma^{1/2} \). However, in our argument we consider the eigenvalues of \( M^{\mathrm{T}} \Sigma M \), not \( \Sigma^{1/2} M^{\mathrm{T}} M \Sigma^{1/2} \). Therefore, the result follows analogously, with the minor modifications to proofs described below. Through the course of the paper, for a right singular vector \( v \) and left singular vector \( u \) of \( \Sigma^{1/2}M \), we often work with \( v \) and \( w = \Sigma^{1/2}u / \left\| \Sigma^{1/2} u\right\| \) rather than \( u \) directly, as it allowed us to leverage the independence of entries of \( M \). In the case where one considers \( M \Sigma^{1/2} \) in place of \( \Sigma^{1/2} M \), one can simply interchange the roles of \( u \) and \( v \). That is, for a singular vector pair \( (u,v) \), our proofs should be modified by replacing \( w \) with \( u \) and replacing \( v \) with \( \Sigma^{1/2} v / \left\| \Sigma^{1/2} v\right\| \). In doing so, one can replace the assumption \( c_{\ref{matrix_model}} \leq \lambda \leq 1 \) with \( c_{\ref{matrix_model}} \leq \lambda \leq c_{\ref{matrix_model}}^{-1} \).
    \end{remark}
\bibliographystyle{abbrv}
\bibliography{singularbib}

\end{document}